\documentclass[12pt]{amsart}

\usepackage{amsmath,amssymb,geometry}
\usepackage[usenames,dvipsnames]{color}

\newtheorem{theorem}{Theorem}[section]

\newtheorem{lemma}{Lemma}[section]
\newtheorem{remark}{Remark}[section]
\newtheorem{definition}{Definition}[section]

\newcommand\E{{E}}
\newcommand\calO{\mathcal O}
\newcommand\PP{{ P}}
\newcommand\Th{{\mathcal T}_h}

\newcommand\nn{\boldsymbol n}
\renewcommand\tt{\boldsymbol t}
\newcommand\tg{t}

\newcommand\grad{\operatorname{grad}}
\newcommand\bgrad{\operatorname{\bf grad}}
\renewcommand\div{\operatorname{div}}
\newcommand\bcurl{\operatorname{\bf curl}}

\newcommand\brot{\operatorname{\bf rot}}
\newcommand\rot{\operatorname{rot}}

\newcommand{\dO}{\,{\rm d}\Omega}
\newcommand{\dE}{\,{\rm d}\E}
\newcommand{\dP}{\,{\rm d}\PP}
\newcommand{\ds}{\,{\rm d}s}
\newcommand{\dx}{\,{\rm d}x}
\newcommand{\dS}{\,{\rm d}S}
\newcommand{\de}{\,{\rm d}e}
\newcommand{\df}{\,{\rm d}f}

\newcommand\bbf{\boldsymbol f}
\newcommand\qq{\boldsymbol q}
\newcommand\rr{\boldsymbol r}
\newcommand\bg{\boldsymbol g}
\newcommand\bpsi{\boldsymbol \psi}
\newcommand\btau{\boldsymbol \tau}
\newcommand\bbeta{\boldsymbol \eta}
\newcommand\bsigma{\boldsymbol \sigma}
\newcommand\bphi{\boldsymbol \varphi}

\newcommand\vphi{\varphi}
\newcommand\vv{\boldsymbol v}

\newcommand\HH{\boldsymbol H}

\newcommand\Vfdk{V^{\rm face}_{2,k}}
\newcommand\Vedk{V^{\rm edge}_{2,k}}
\newcommand\Vftk{V^{\rm face}_{3,k}}
\newcommand\Vetk{V^{\rm edge}_{3,k}}
\newcommand\Vettk{\widetilde{V}^{\rm edge}_{3,k}}   
\newcommand\Wetk{W^{\rm edge}_{3,k}}                      
\newcommand\Rort{\calR_{k}\slash\calR_{k-2}}

\newcommand{\Vhelem}[1]{V^{\rm elem}_{3,#1}}

\newcommand{\Vhv}{V^{\rm vert}_{3,k}}

\newcommand{\calF}{ \mathcal{F}}
\newcommand{\calB}{ \mathcal{B}}
\newcommand{\calG}{ \mathcal{G}}
\renewcommand{\O}{ {\mathcal O}}
\newcommand{\calP}{ \mathcal{P}}
\newcommand{\calR}{ \mathcal{R}}
\newcommand{\calS}{ \mathcal{S}}
\newcommand{\dofop}{{\mathcal D}}
\newcommand\R{\mathbb{R}}
\renewcommand{\P}{ {\mathbb P}}

\newcommand\Pit{\widetilde{\Pi}}

\newcommand{\myarrow}[1]{\xrightarrow[\hspace*{0.5cm}]{\raisebox{1mm}{\normalsize$#1$}}}
\newcommand{\shortarrow}[1]{\xrightarrow[\hspace*{0.4cm}]{\raisebox{1mm}{\normalsize$#1$}}}
\newcommand{\vshortarrow}[1]{\xrightarrow[\hspace*{0.1cm}]{\raisebox{1mm}{\normalsize$#1$}}}

\numberwithin{equation}{section}

\definecolor{dgreen}{rgb}{0,0,0}
\definecolor{dred}{rgb}{0,0,0}
\definecolor{dblue}{rgb}{0,0,0}

\begin{document}

\begin{center}
\begin{LARGE}
{\bf $H(\div)$ and $H({\bf curl})$-conforming VEM}
\end{LARGE}

\vskip1.5cm

\noindent
L. Beir\~ao da Veiga
\footnote{Dipartimento di Matematica, Universit\`a di Milano,
Via Saldini 50, 20133 Milano (Italy),
and IMATI del CNR, Via Ferrata 1, 27100 Pavia, (Italy)},
F. Brezzi
\footnote{IUSS, Piazza della Vittoria 15, 27100 Pavia (Italy),
and IMATI del CNR, Via Ferrata 1, 27100 Pavia, (Italy)}, 
L.~D.~Marini
\footnote{Dipartimento di Matematica, Universit\`a di Pavia,
and IMATI del CNR, Via Ferrata 1, 27100 Pavia, (Italy)},  
and A. Russo
\footnote{Dipartimento di Matematica e Applicazioni, Universit\`a di Milano-Bicocca,
via Cozzi 57, 20125 Milano (Italy)
and
IMATI del CNR, Via Ferrata 1, 27100 Pavia, (Italy)}

\end{center}





\vskip1.5cm

\begin{center}
{\bf Abstract}
\end{center}
In the present paper we construct Virtual Element Spaces that are $H({\rm div})$-conforming and $H({\rm \bf curl})$-conforming on general polygonal and polyhedral elements;
these spaces can be interpreted as a generalization of well known Finite Elements. 
We moreover present the basic tools needed to make use of these spaces in the approximation of partial differential equations.  
Finally, we discuss the construction of exact sequences of VEM spaces. 

\vskip1cm


\section{Introduction}
The Virtual Element Methods where initially introduced in \cite{volley}, as a variant of classical Lagrange Finite Element Methods to accommodate the use of polygonal and polyhedral elements. Needless to say, they could be seen as an evolution of nodal Mimetic Finite Differences (see \cite{Brezzi:Buffa:Lipnikov:2009,BLM11}) as well as a variant of other Galerkin methods for polygonal and polyhedral elements (see e.g.{\color{dblue} \cite{Arroyo-Ortiz,MESHLESS 12.,GFEM 13.,MESHLESS 16.,Bishop,Bonelle:Ern:2014,JA12,FHK06,Fries:Belytschko:2010,GFEM 92.,GFEM 93.,Kuznetsov:Repin:2003,POLY [28],Sukumar:Malsch:2006,Tabarraei:Sukumar:2007,TPPM10,TPPM12,Wachspress75,Wachspress11}} and the references therein). Even more recently, in \cite{BFM-mixed} we started the extension to polygonal elements of Raviart-Thomas or BDM elements for mixed formulations (see e.g. \cite{Bo-Bre-For} and the references therein). These, in a sense, constitute the most natural and direct evolution of the original ``flux based'' Mimetic Finite Differences, as for instance in \cite{Hyman:Shashkov:1997b}. See also, for the more mathematical aspects, \cite{Brezzi:Lipnikov:Shashkov:2005,Brezzi:Lipnikov:Simoncini:2005,Brezzi:Lipnikov:Shashkov:Simoncini:2007}, as well as
\cite{BeiraodaVeiga:Manzini:2008b,Gyrya:Lipnikov:2008,BeiraoDaVeiga:Lipnikov:Manzini:2009}, the review papers \cite{Bochev:Hyman:2006,Droniou:Eymard:Gallouet:Herbin:2010,MFD[22]}, and the book \cite{MFD[23]}.
{\color{dred}
In addition to \cite{volley}, see for instance \cite{Brezzi:Marini:plates,VEM-elasticity,projectors,BM13,BFM-mixed} and references therein for {\color{dblue}applications of the Virtual Element Method to various types of problems}
}

On the other hand, to deal with a sufficiently wide range of mixed formulations (see again \cite{Bo-Bre-For} and the references therein), one needs to use a big variety of $H({\rm div})$ and
$H({\rm \bf curl})$-conforming spaces (to be used together with the more classical $H^1$-conforming and  $L^2$-conforming ones). See for instance \cite{Matiussi:1997} or \cite{AFW-Acta}.  See also the recent overview on Finite Element spaces presented in \cite{Periodic}.

The purpose of this paper is to indicate a possible strategy to construct the extensions of all these types of spaces to more general elemental geometries, and typically to polygonal and polyhedral elements. The use of curved edges or curved faces (that so far, in this context, was tackled only in \cite{Brezzi:Lipnikov:Shashkov:2006}) will be the object of future research.

As a general matter, the (vector valued) functions to be used, in each element, in the Virtual Element Methods are {\bf not} polynomials (although they contain suitable polynomial spaces within each element),
and are presented as solutions of (systems of)  partial differential equations.
However ``the name of the game'', in the VEM context, is to avoid solving these PDE systems, even in a roughly approximate way.  Hence, in order to be able to construct, element by element, the necessary local
matrices, we have to be able to construct suitable {\it projectors} from the local  VEM spaces to some polynomial spaces (whose degree will determine the final accuracy of the method).

In presenting our $H({\rm div})$-conforming and $H({\rm \bf curl})$-conforming {\color{dblue}spaces} we will therefore take care to show how, for {\color{dblue}them}, one can construct suitable $L^2$-projection operators on the corresponding
polynomial spaces. This of course will not always solve all the problems, but (as pointed out for instance in \cite{projectors} for some particular cases) will surely be a precious instrument.

As the variety of possible {\color{dblue}variants (required by different applications)} is overwhelming, we decided to limit ourselves, here, to the
presentation of a few typical cases (that in our opinion could be sufficient to give the general idea),
leaving to the very last (and short) section the task to give hints on {\color{dblue} some  of the possible variants}. In the same spirit, we decided not to present direct applications. We believe that, for the readers with some experience in the approximation of mixed formulations, the general ideas outlined in this paper should  be enough to understand the possible use of our
spaces {\color{dblue} for} most of the applications discussed in \cite{Bo-Bre-For}. Clearly, a lot of additional work,
and a lot of numerical experiments, will be needed for the tune-up of these methods in each particular type of application.
{\color{dblue}
To have an idea on the implementation of Virtual Element Methods we refer to the guidelines given in \cite{hitchhikers} for nodal virtual elements.
}

Here is an outline of the paper: in the next section we will introduce a suitable notation and recall
a few classical results of Calculus in several variables. Then we will present, each in a separate section,
the $H({\rm div})$-conforming and {\color{dblue} the} $H({\rm curl})$-conforming spaces for {\it  polygonal} elements, and the {\color{dblue} corresponding} ones for {\it polyhedral } elements. Next, we will briefly
recall the $H^1$-conforming and $L^2$-conforming spaces (as introduced for instance in \cite{volley}) and
discuss the possibility of having {\it exact sequences} of VEM spaces, in the spirit of \cite{AFW-Acta}.

In the last section, as announced already, we will give a short  hint of the huge variety of possible variants.

\section{Notation, Assumptions, and Known Results}\label{nota}


{\color{dblue}
In what follows, we will detail the spaces and their degrees of freedom mainly {\bf at the element level}. One of the best features of Virtual Element Methods is the possibility to use elements having a very general geometry, and actually, in order to give the definition of the space we could use arbitrary {\it simply connected polygons and polyhedra}. In order to have optimal {\it interpolation errors}, as well as suitable stability properties in the applications to different problems, we would however need some mild assumptions. Here below, we give the {\it flavor} of the type of assumptions that are generally used in Virtual Element Methods.

In  two dimensions we will assume that we deal with a polygon $\E$
having $\ell_e$ edges and containing a disk $D_{\E}$ such that  $\E$ is star-shaped
with respect to all the points of $D_{\E}$.

In three dimensions we will assume that we are dealing with
a polyhedron $\PP$ having $\ell_e$ edges and  $\ell_f$ faces, containing a ball $B_{\PP}$ such that
$\PP$ is star-shaped with respect to all points of  $B_{\PP}$. We will also assume that each
face $f$  is star-shaped with respect to all the points of a disk $D_f$.

 Note that under all these
assumptions both $\E$, $\PP$, and each face of $\PP$ will have to be {\it simply connected}.

Actually it will not be a problem to use elements that are suitable unions of pieces that satisfy the assumptions above.

Needless to say, when dealing with a {\it sequence of decompositions} $\{\Th\}_h$, one should make some further assumptions. To start with, for every geometrical object $\calO$ that we are going to use in what follows (edge, face,
element, etc.) we will denote its {\it diameter} by $h_{\O}$. Then} we assume that there exists a positive number $\kappa$ such that, in two dimensions, for every decomposition ${\mathcal T}_h$ and for every element $\E$ in ${\mathcal T}_h$, we have:
\begin{itemize}
\item   $h_{D_{\E}}\ge\kappa h_{\E}$
\item for every edge $e$ of $\E$ we have $h_e\ge\kappa h_{\E}$
\end{itemize}
and in three dimensions, for every decomposition ${\mathcal T}_h$
and for every element $\PP$ in ${\mathcal T}_h$, we have
\begin{itemize}
\item   $h_{B_{\PP}}\ge\kappa h_{\PP}$
\item for every  face $f$ of $\PP$  we have $h_{D_f}\ge\kappa h_{\PP}$
\item for every  face $f$ of $\PP$ and for every edge $e$ of $f$: $h_e\ge\kappa h_{f}$.
\end{itemize}
 We point out that these additional assumptions will imply, in particular, that there exists an
 integer number $K$ depending only on $\kappa$ such that every element has less than $K$ faces
 and each face (and each element in two dimensions) has less than $K$ edges.

 {\color{dblue} We also point out that the above assumptions, that are indeed quite general, are very likely unnecessarily restrictive. Indeed, from our numerical experiments, these methods show a remarkable robustness, allowing for instance polygons with edges that are arbitrarily small  compared with the
 diameter of the element itself. We consider however that he present generality is sufficient in almost
 every practical case, and we decided, for the moment, to avoid unnecessary technical complications
 in order to increase it.}

\medskip

Here below we introduce now  some additional notation.
 \begin{itemize}
 \item For a space of functions $\calF(\O)$ defined on $\O$, we denote by $\calF(\O){/\R}$ (or simply by $\calF{/\R}$ when the context is clear) the subset of functions having zero mean value on $\O$.
     \end{itemize}
\begin{itemize}
\item {\color{dblue}\bf In two dimensions,} for a  polygon $\E$, $\nn_{\E}$ or simply $\nn$ will be the outward normal unit vector, and $\tt_{\E}$, or simply $\tt$,  will be the tangent counterclockwise unit vector.
    \item {\color{dblue} For} a scalar field $q$ and a vector field ${\vv}=(v_1,v_2)$, we will set (with a usual notation)
\begin{equation*}
\brot q:=\Big(\frac{\partial q}{\partial y},-\frac{\partial q}{\partial x}\Big)
\qquad \rot{\vv}:=\frac{\partial v_2}{\partial x}-\frac{\partial v_1}{\partial y}
\end{equation*}
\end{itemize}
\begin{itemize}
\item {\color{dblue} \bf In three dimensions,} for a face $f$ of a polyhedron $\PP$, the {\bf tangential} differential operators  will
 be denoted by a subscript 2, as in: $\div_2$, $\rot_2$, ${\brot_2}$, $\bgrad_2$, $\Delta_2$, and so on.

 \item {\color{dblue} When dealing with a single polyhedron, we will always assume that all its faces are oriented
       with the {\it outward normal}, while, when necessary, we will have to choose an orientation
       for every edge. Obviously when dealing with a decomposition in several polyhedra
       we will also have to decide an orientation for every face.}

  \item  On a polyhedron $\PP$, on each face $f$ we will have to distinguish between the unit
 outward {\it normal to the plane of the face} (that we denote by $\nn_{\PP}^f$), and the unit vector {\it in the plane of the face} that is normal to the boundary $\partial f$ (that will be denoted, on each edge $e$, by $\nn_f^e$). On each face, $\tt_f$ or simply $\tt$ will again be the unit counterclockwise tangent vector on $\partial f$.

 \item For a (smooth enough) three dimensional vector-valued function $\bphi$ on $\PP$, and for a
face $f$ with normal $\nn_{\PP}^f$,
we define the tangential component of $\bphi$ as
\begin{equation}\label{partetan}
\bphi_f:=\bphi-(\bphi\cdot\nn_{\PP}^f)\nn_{\PP}^f,
\end{equation}
while $\bphi_{\tg}$ denotes the vector field defined on $\partial\PP$ such that, on each face $f\in\partial\PP$, its restriction to the face $f$ satisfies:
\begin{equation}\label{deftg}
\bphi_{\tg|f}=\bphi_f .
\end{equation}
\item Note that $\bphi_f$ as defined in \eqref{partetan} is different from
\begin{equation}\label{wedgenorm}
\bphi\wedge\nn_{\PP}^f;
\end{equation}
indeed, for instance, if $\nn_{\PP}^f=(0,0,1)$ and $\bphi=(\phi_1,\phi_2,\phi_3)$, then
\begin{equation*}
\bphi_f=(\phi_1,\phi_2,0)\qquad \bphi\wedge\nn_{\PP}^f=(\phi_2,-\phi_1,0).
\end{equation*}
\item With an abuse
of language, sometimes we will treat both $\bphi_f$ and $\bphi\wedge\nn_{\PP}^f$ as 2-d vectors in the plane of the face. In the previous case, then, we would often {\color{dblue} take} $\bphi_f=(\phi_1,\phi_2)$ and $\bphi\wedge\nn_{\PP}^f=(\phi_2,-\phi_1)$.

\end{itemize}

We will now recall some basic properties of Calculus of several variables, applied in particular
to polynomial spaces. Before doing that, we recall {\color{dblue} once more} that our assumptions imply that all our elements are simply connected.

 For a generic non negative number $k$ and for a generic geometrical object $\O$
 in 1,2, or 3 dimensions we will denote
 \begin{itemize}
\item $\P_{k}(\O)=\mbox{ Polynomials of degree $\leq k$ on }\O$,
 \end{itemize}
 with the additional (common) convention that
  \begin{itemize}
\item $\P_{-1}(\O)=\{0\}$.
 \end{itemize}
Moreover, with a common abuse of language, we will often say
``polynomial of degree $k$'' meaning
actually ``polynomial of degree $\le k$''.
Often the geometrical object $\O$ will be omitted when no confusion arises.

In all the following diagrams \eqref{gradrotex}, \eqref{rotdivex}, and
\eqref{gradrotdivex}, {\color{dblue} as well as in the ones at the end, as \eqref{2D:virt:seq1:k},
\eqref{2D:virt2:seq:k}, and \eqref{3D:virt:seq:k}}
we will denote by $i$ the mapping that to every real {\it number}
$c$ associates the constant {\it function} identically equal to $c$, and
by $o$ the mapping that to every {\it function} associates the {\it number} $0$.
Then we recall that, in 2 and in 3 dimensions, we have the exactness of the following
sequences.

\noindent In  2 dimensions
\begin{equation}\label{gradrotex}
\R \myarrow{i} \P_r \myarrow{\bgrad} (\P_{r-1})^2 \myarrow{\rot} \P_{r-2} \myarrow{o} 0
\end{equation}
or, equivalently,
\begin{equation}\label{rotdivex}
\R \myarrow{i} \P_r \myarrow{\brot} (\P_{r-1})^2 \myarrow{\div} \P_{r-2} \myarrow{o} 0
\end{equation}
are exact sequences. In three dimensions we have that
\begin{equation}\label{gradrotdivex}
\R \myarrow{i} \P_r \myarrow{\bgrad} (\P_{r-1})^3
\myarrow{\bcurl} (\P_{r-2})^3 \myarrow{\div} \P_{r-3} \myarrow{o} 0
\end{equation}
is also an exact sequence. We recall that the {\it exactness}
means that {\it the image of every
operator coincides with the kernel of the following one}. To better explain the consequences of these statements  we introduce an additional notation. For $s$ integer $\ge 1$, in two dimensions we denote by
\begin{itemize}
\item  $\calG_{s-1}$ the set $\bgrad(\P_s)$,
\item  $\calR_{s-1}$ the set $\brot(\P_{s})$,
\end{itemize}
and in three dimensions
\begin{itemize}
\item  $\calG_{s-1}$ the set $\bgrad(\P_s)$,
\item  $\calR_{s-1}$ the set $\bcurl\Big((\P_{s})^3\Big)$.
\end{itemize}
If we are considering polynomials on a domain $\O$ (not disgustingly irregular) we
might use the $L^2(\O)$ or (in $d$ dimensions) the  $(L^2(\O))^d$ inner product, and introduce
\begin{itemize}
\item
$\calG^{\perp}_{s}$
as the {\it orthogonal} of $\calG_{s}$ in $(\P_{s})^d$,
\item and $\calR^{\perp}_{s}$
as the {\it orthogonal} of $\calR_{s}$ in $(\P_{s})^d$.
\end{itemize}
Obviously, $(\P_{s})^d=\calG_{s}\oplus \calG^{\perp}_{s}=\calR_{s}\oplus\calR^{\perp}_{s}$.
In a similar way, the space $\P_s$ could
be seen as decomposed in the subspace of constants (the image of  $i:\R\longrightarrow \P_s$) and the polynomials
in $\P_s$ having zero mean value on $\O$ (and hence orthogonal to the constants), that is  $(\P_s(\O)){/\R}$.

We recall now some of the properties following from the exactness of the above sequences.
The exactness of the sequence
\eqref{gradrotex} implies in particular that for all integer $s$:
\begin{equation}\label{ex1}
\begin{aligned}
&\mbox{i) $\bgrad$ is an isomorphism from $(\P_s){/\R}$ to $\calG_{s-1}$},\\
&\mbox{ii) $\{\vv\in(\P_{s})^2\}\Rightarrow \{\rot\vv=0$ iff $\vv\in\calG_{s}\}$},\\
&\mbox{iii) $\rot$ is an isomorphism from $\calG_{s}^\perp$ to the whole $\P_{s-1}$},
\end{aligned}
\end{equation}
and equivalently \eqref{rotdivex} implies that
\begin{equation}\label{ex2}
\begin{aligned}
&\mbox{i) $\brot$ is an isomorphism from $(\P_s){/\R}$ to $\calR_{s-1}$},\\
&\mbox{ii)$\{\vv\in(\P_{s})^2\}\Rightarrow \{\div\vv=0$ iff $\vv\in\calR_{s}\}$},\\
&\mbox{iii) $\div$ is an isomorphism from $\calR_{s}^\perp$ to the whole $\P_{s-1}.$}
\end{aligned}
\end{equation}
Finally, the exactness of the sequence \eqref{gradrotdivex} implies in particular that,  for all integer $s$:
\begin{equation}\label{ex3}
\begin{aligned}
&\mbox{i) $\{\vv\in(\P_{s})^3\}\Rightarrow \{\bcurl\vv=0$ iff $\vv\in\calG_{s}\}$},\\
&\mbox{ii)$\{\vv\in(\P_{s})^3\}\Rightarrow \{\div\vv=0$ iff $\vv\in\calR_{s}\}$},\\
&\mbox{iii) $\bgrad$ is an isomorphism from $(\P_s){/\R}$ to $\calG_{s-1}$},\\
&\mbox{iv) $\bcurl$ is an isomorphism from $\calG_{s}^\perp$ to $\calR_{s-1}$},\\
&\mbox{v) $\div$ is an isomorphism from $\calR_{s}^{\perp}$ to the whole $\P_{s-1}$}.
\end{aligned}
\end{equation}

\begin{remark}\label{calculus} Properties [\ref{ex1};ii)],  [\ref{ex2};ii)], and  [\ref{ex3};i) and ii)] are just particular cases of well known results in Calculus. Indeed, on a simply connected
domain, we know that a (smooth enough) vector field $\vv$ having $\rot\vv=0$ (in 2 dimensions) or $\bcurl\vv=0$ (in 3 dimensions) is necessarily a gradient, and a (smooth enough) vector $\vv$ field having $\div\vv=0$  is necessarily a $\brot$ (in 2 dimensions) or a $\bcurl$
(in 3 dimensions).
\end{remark}

\noindent To all these spaces we can attach their dimensions.  {\color{dblue} To start with, } we  denote
by $\pi_{k,d}$ the dimension of the space $\P_k(\R^d)$, that is,
\begin{equation}\label{dimpi}
\pi_{k,1}=k+1;\ \pi_{k,2}=\dfrac{(k+1)(k+2)}{2};\ \pi_{k,3}=\dfrac{(k+1)(k+2)(k+3)}{6}.
\end{equation}
We then consider the spaces of (vector valued) polynomials $(\P_k)^d$ whose dimension is obviously
\begin{equation}\label{dimbpi}
\dim\{(\P_k)^d\}=d\pi_{k,d}.
\end{equation}
Among them, we consider those that are {\it gradients} (that we already called $\calG_k$), and we denote by $\gamma_{k,d}$ their dimension:
\begin{equation}\label{dimG}
\dim\{\calG_k\}\mbox{ in $d$ dimensions }\equiv\,\gamma_{k,d}=\pi_{k+1,d}-1.
\end{equation}
Needless to say, $\gamma_{k,2}$ also equals the dimension $\rho_{k,2}$ of $\brot(\P_{k+1})$ (that is,
$\calR_k$ in two dimensions):
\begin{equation}
\dim\{\calR_k\}\mbox{ in $2$ dimensions }=\rho_{k,2}=\gamma_{k,2}=\pi_{k+1,2}-1.
\end{equation}

 \noindent We also have (obviously), in $d$ dimensions,
\begin{equation}\label{dimGperp}
\dim\{\calG_k^{\perp}\}\mbox{ in $d$ dimensions }=d\pi_{k,d}-\gamma_{k,d}=
d\pi_{k,d}-\pi_{k+1,d}+1.
\end{equation}
In 2 dimensions, looking at [\ref{ex1};iii)] {\color{dblue} and at [\ref{ex2};iii)] we see that the dimension of $\calG_{k}^{\perp}$ as well as that of $\calR_{k}^{\perp}$
equal that of $\P_{k-1}$, that is
\begin{equation} \label{RG2d}
\dim\{\calG_{k}^{\perp}\}=\dim\{\calR_{k}^{\perp}\}=\pi_{k-1,2}
\qquad\mbox{in two dimensions.}
\end{equation}
}

On the other hand, for $d=3$, we can use [\ref{ex3};iv)] and see
that  the dimension
$\rho_{k-1,3}$ of $\calR_{k-1}=\bcurl((P_{k})^3)$ is given by
\begin{equation}\label{dimGperp3}
\rho_{k-1,3}=\dim\{\calR_{k-1}\}=\dim\{\calG_k^{\perp}\}=3\pi_{k,3}-\pi_{k+1,3}+1,
\end{equation}
while, following [\ref{ex3};v], {\color{dblue} we have
\begin{equation} \label{RG3d}
\dim\{\calR_{k}^{\perp}\}=\pi_{k-1,3}\qquad\mbox{in three dimensions}
\end{equation}

We summarize all the above results on the dimensions of polynomial spaces in the following
equations. In {\bf two dimensions}:
\begin{equation} \label{in2dim}
\dim\{\calG_{k}\}=\dim\{\calR_{k}\}= \pi_{k+1,2}-1
\qquad\dim\{\calR_{k}^{\perp}\}=\dim\{\calG_{k}^{\perp}\}=\pi_{k-1,2}
\end{equation}
 and in {\bf three dimensions}:
\begin{multline} \label{in3dim}
\dim\{\calG_{k}\}=\pi_{k+1,3}-1, \qquad\dim\{\calG_{k}^{\perp}\}=3\pi_{k,3}-\pi_{k+1,3}+1\\
\dim\{\calR_{k}\}= 3\pi_{k+1,3}-\pi_{k+2,3}+1\qquad
\qquad\dim\{\calR_{k}^{\perp}\}=\pi_{k-1,3}.
\end{multline}
}

{\color{dgreen} As announced, the definition of our local Virtual Element spaces will be done as the solution, within each element, of a suitable {\it div-curl} system. In view of that, it will be convenient to recall the {\it compatibility conditions} (between the data inside the element and the ones at the boundary) that are required in order to have a solution.  To start with, for a polygon $\E$ we define
\begin{equation}\label{defHd2}
H(\div;\E):=\{\vv\in (L^2(\E))^2 \mbox{ such that }\div\vv\in L^2(\E)\},
\end{equation}
\begin{equation}\label{defHr2}
H(\rot;\E):=\{\vv\in (L^2(\E))^2 \mbox{ such that }\rot\vv\in L^2(\E)\},
\end{equation}
and for a polyhedron $\PP$
\begin{equation}\label{defHd3}
H(\div;\PP):=\{\vv\in (L^2(\PP))^3 \mbox{ such that }\div\vv\in L^2(\PP)\},
\end{equation}}
{\color{dgreen}
\begin{equation}\label{defHr3}
H(\rot;\PP):=\{\vv\in (L^2(\PP))^3 \mbox{ such that }\bcurl\vv\in (L^2(\PP))^3\}.
\end{equation}
We now assume that we are given, on a simply connected polygon $\E$, two smooth functions $f_d$ and
$f_r$, and, on the boundary $\partial\E$, two edge-wise smooth functions $g_n$ and $g_t$.
We recall that the problem: {\it find $\vv\in H(\div;\E)\cap H(\rot;\E)$ such that:}
\begin{equation}
\div\vv=f_d \mbox{ and }\rot\vv=f_r\mbox{ in }\E\quad\mbox{ and }\quad \vv\cdot\nn=g_n \mbox{ on }
\partial\E
\end{equation}
has a unique solution if and only if
\begin{equation}
\int_{\E}\div\vv\dE = \int_{\partial\E} g_n\ds.
\end{equation}
Similarly the problem: {\it find $\vv\in H(\div;\E)\cap H(\rot;\E)$ such that:}
\begin{equation}
\div\vv=f_d \mbox{ and }\rot\vv=f_r\mbox{ in }\E\quad\mbox{ and }\quad \vv\cdot\tt=g_t \mbox{ on }
\partial\E
\end{equation}
has a unique solution if and only if
\begin{equation}
\int_{\E}\rot\vv\dE = \int_{\partial\E} g_t\ds.
\end{equation}
In three dimension, on a simply connected polyhedron $\PP$ we assume that we are given a smooth scalar function
$f_d$ and a smooth vector valued function $\bbf_r$ with $\div\bbf_r=0$. On the boundary $\partial\PP$
we assume that we are given a face-wise smooth scalar function $g_n$ and a face-wise smooth tangent
vector field $\bg_t$ whose tangential components are continuous (with a natural meaning) at the edges
of $\partial\PP$. Then we recall that
the problem: {\it find $\vv\in H(\div;\PP)\cap H(\bcurl;\PP)$ such that:}
\begin{equation}
\div\vv=f_d \mbox{ and }\bcurl\vv=\bbf_r\mbox{ in }\PP\quad\mbox{ and }\quad \vv\cdot\nn=g_n \mbox{ on }
\partial\PP
\end{equation}
has a unique solution if and only if
\begin{equation}
\int_{\PP}\div\vv\dP = \int_{\partial\PP} g_n\ds,
\end{equation}
and similarly the problem: {\it find $\vv\in H(\div;\PP)\cap H(\bcurl;\PP)$ such that:}
\begin{equation}
\div\vv=f_d \mbox{ and }\bcurl\vv=\bbf_r\mbox{ in }\PP\quad\mbox{ and }\quad \vv_t=\bg_t \mbox{ on }
\partial\PP
\end{equation}
has a unique solution if and only if
\begin{equation}
\bbf_r\cdot\nn = \rot_2 \bg_t \mbox{ on }\partial\PP.
\end{equation}
For more details concerning the solutions of the {\it div-curl system} we refer, for instance, to
\cite{Auchmuty:2d}, \cite{Auchmuty:3d} and the references therein.

Finally, in order to help the reader to understand what we consider as {\it feasible} (in a code),} we recall that we assume to be able to integrate any polynomial on any polygon or polyhedron, for instance through formulae of the type
\begin{equation}\label{intpol}
\int_{\E}x_1^k=\frac{1}{k+1}\int_{\partial\E}x_1^{k+1}\,n_1\ds.
\end{equation}

%
%
%
%
%
%
%
%
%
%
%
%
%
%
%
%

\section{2D Face Elements}

These spaces are the same of Brezzi-Falk-Marini \cite{BFM-mixed}, although here we propose a different
set of degrees of freedom.

\subsection{ The local space}

On a polygon $\E$, {\color{dblue} for $k$ integer $\ge 1$,} we set:
\begin{multline}\label{Vface2d}
\Vfdk(\E):=\{\vv\in H(\div;\E)\cap H(\rot;\E)  :
\vv\cdot\nn_{|e}\in\P_{k}(e)~\forall \mbox{ edge $e$ of } \E,\\
\bgrad\div\vv\in\calG_{k-2}(\E), \mbox{ and }
\rot\vv\in\P_{k-1}(E)\}.
\end{multline}

\subsection{Dimension of the space $\Vfdk(\E)$}
{\color{dgreen}
We recall from our introduction that, given
 \begin{itemize}
\item a function $g$ defined on $\partial\E$ such that $g_{|e}\in\P_{k}(e)$ for all $e \in \partial E$,
\item a polynomial $f_d\in\P_{k-1}(\E)$ such that
\begin{equation}\label{compadd}
\int_{\E}f_d\dE=\int_{\partial\E} g \ds ,
\end{equation}
\item a polynomial $f_r\in\P_{k-1}(\E)$ ,
\end{itemize}
we can find a unique vector $\vv\in\Vfdk(\E)$ such that
\begin{equation}\label{perdimfbd}
\vv\cdot\nn=g \mbox{ on }\partial\E, \quad \div\vv=f_d\mbox{ in }\E, \quad \rot\vv=f_r\mbox{ in }\E.
\end{equation}}
 This easily implies that the dimension of $\Vfdk(\E)$ is given by:
\begin{equation}\label{dimfdd}
\begin{aligned}
\dim\Vfdk(\E) & =\ell_e\dim\P_{k}(e)+\{\dim \P_{k-1}(\E)-1\}+\dim\P_{k-1}(\E) \\
& =\ell_e\pi_{k,1}+\pi_{k-1,2}-1+\pi_{k-1,2}
\end{aligned}
\end{equation}

\begin{remark}\label{ttraces} We note that, for a vector-valued function in $H(\div;\E)\cap H(\rot;\E)$, one can define both the normal and the tangential trace
{\it on each edge} of $\partial\E$ (see \cite{Costabel}).
\end{remark}

\subsection{ The Degrees of Freedom}

A convenient set of degrees of freedom for functions  $\vv$  in $\Vfdk(\E)$ will be:
\begin{eqnarray}\label{doffdd0}
&\int_e{\vv\cdot\nn}\,{p}_{\,k}\,\de
&\quad \mbox{ for all edge $e$, for all }\;
p_{k}\in\P_k(e),\label{dof1}\\[3pt]
&\int_{\E}{\vv\cdot \bg_{k-2}}\dE &\quad\mbox{ for all $\bg_{k-2}\in\calG_{k-2}$},  \label{dof2}
\\[3pt]
&\int_{\E}{\vv\cdot \bg_{k}^\perp}\dE  &\quad \mbox{ for all $\bg_{k}^\perp\in\calG_{k}^\perp$}
.\label{dof3}
\end{eqnarray}
{\color{dblue}Remembering \eqref{in2dim} we easily see that  number of degrees of freedom \eqref{dof1}--\eqref{dof3} equals the dimension of $\Vfdk(\E)$ as given in \eqref{dimfdd}.}

\subsection{ Unisolvence}

Since the number of degrees of freedom
\eqref{dof1}-\eqref{dof3} equals the dimension of $\Vfdk(\E)$, to prove unisolvence  we just need to show that
if for a given $\vv$ in $\Vfdk(\E)$ all the degrees of freedom \eqref{dof1}-\eqref{dof3} are zero,
that is if
\begin{eqnarray}\label{doffdd}
&\int_e{\vv\cdot\nn}\,{p}_{\,k}\,\de=0&
\quad \mbox{ for all edge $e$, for all }\;
p_{k}\in\P_k(e),\label{dof10}\\[3pt]
&\int_{\E}{\vv\cdot \bg_{k-2}}\dE=0 &\quad\mbox{ for all $\bg_{k-2}\in\calG_{k-2}$},  \label{dof20}
\\[3pt]
&\int_{\E}{\vv\cdot \bg_{k}^\perp}\dE=0  &\quad \mbox{ for all $\bg_{k}^\perp\in\calG_{k}^\perp$},
\label{dof30}
\end{eqnarray}
then we must have $\vv=0$. For this we introduce a couple of preliminary observations.
\begin{lemma}\label{primi0}
If $\vv\in \Vfdk(\E)$ and if \eqref{dof10} and \eqref{dof20} hold, then
\begin{equation}\label{orthograd2}
\int_{\E} \vv\cdot\bgrad\varphi\dE=0\quad
\forall\varphi\in H^1(\E).
\end{equation}
\end{lemma}
\begin{proof} {\color{dblue} Using the fact that
 $\div\vv\in\P_{k-1}$ and setting $q_{k-1}:=\div\vv$ we have
\begin{multline}\label{added:X1}
\int_{\E}|\div\vv|^2\dE=\int_{\E}\div\vv\,q_{k-1}\dE
=\int_{\partial\E}\vv\cdot \nn q_{k-1}\ds-\int_{\E}\vv \cdot \bgrad q_{k-1}\dE
=0,
\end{multline}
 where the last step follows from \eqref{dof10} and \eqref{dof20}. Hence we
 have that $\div\vv=0$ and since (using again \eqref{dof10}) $\vv\cdot \nn=0$
 on $\partial\E$, the result \eqref{orthograd2} follows then using a simple integration by parts.}
\end{proof}

\begin{lemma}\label{deco1}
If $\vv\in \Vfdk(\E)$ then there exist a $\qq_{k}^\perp$ in $\calG_k^\perp$  and a $\varphi\in H^1(\E)$
such that
\begin{equation}
\vv=\qq_k^\perp+\bgrad\varphi .
\end{equation}
\end{lemma}
\begin{proof} We first note that according to \eqref{Vface2d} if  $\vv\in \Vfdk(\E)$ then
$\rot\vv\in\P_{k-1}$. Looking at [\ref{ex1};iii)]  we have then that
$\rot\vv=\rot\qq_{k}^\perp$ for some $\qq^{\perp}_{k}\in\calG_k^\perp$. Now the difference $\vv-\qq_k^\perp$
satisfies $\rot(\vv-\qq_k^\perp)=0$, and as $\E$ is simply connected the result follows
from Remark \ref{calculus}.
\end{proof}

We can now easily prove the following theorem.
\begin{theorem} The degrees of freedom \eqref{dof1}-\eqref{dof3} are unisolvent in $\Vfdk(\E)$.
\end{theorem}
\begin{proof} Assume that for a certain $\vv\in\Vfdk(\E)$ we have  \eqref{dof10}-\eqref{dof30}.
From Lemma \ref{deco1} we have $\vv= \qq_k^\perp+\grad\varphi$ for some $\qq\in{\color{dblue}\calG}_{k}^{\perp}$
and some $\varphi\in H^1(\E)$. Then
\begin{equation}\label{mungi}
\int_{\E}|\vv|^2\dE=\int_{\E}\vv\cdot(\qq_k^\perp+\bgrad\varphi)\dE=0
\end{equation}
since the first term is zero by \eqref{dof30} and the second term is zero by \eqref{dof10}-\eqref{dof20}
and Lemma \ref{primi0}.
\end{proof}

\begin{remark}\label{alternativedof} The degrees of freedom \eqref{dof1} are pretty obvious. A natural variant would be to use, on each edge $e$, the values of $\vv\cdot\nn$ at the $k+1$ Gauss points
on $e$.
On the other hand, for the degrees of freedom \eqref{dof2} we could integrate by parts,
and substitute them with
\begin{equation}
\int_{\E}\div\vv\,q_{k-1}\dE \qquad \mbox{for all } q_{k-1}\in\P_{k-1}/\R .
\end{equation}
Finally, the degrees of freedom \eqref{dof3} could be replaced by
\begin{equation}
\int_{\E}\rot\vv\,q_{k-1}\dE \qquad \mbox{for all } q_{k-1}\in\P_{k-1}
\end{equation}
as we had in the original work \cite{BFM-mixed}.
\end{remark}
\begin{remark}\label{otherdoff2} Needless to say, certain degrees of freedom will be more convenient
when writing the code, and others might be more convenient when writing a proof. For instance, from the above discussion it is pretty obvious that we can identify uniquely an element $\vv$ of $\Vfdk(\E)$
by prescribing its normal component
$\vv\cdot\nn$ (in $\P_k(e)$) on every edge, its rotation $\rot\vv$ (in $\P_{k-1}(\E)$), and its divergence
$\div\vv$ (in $(\P_{k-1}(\E))/{\R}$), provided the compatibility condition \eqref{compadd}
is satisfied. {\color{dblue} This will be convenient in some proof, but might be less convenient in the code.}
\end{remark}

\subsection{Computing the $L^2$ projection}
\label{computing:face:2d}

Since the VEM spaces contain functions which are not polynomials, and their reconstruction can be too hard,
for the practical use of a virtual element method it is {\color{dblue} often}
important to be able to compute different types of projections onto spaces of polynomials.
Here we show how to construct the one that is possibly the most convenient, and surely the most
 commonly used: the $L^2$ projection onto $(\P_{k}(\E))^2$.

For this, we begin by recalling that to assign
$\bgrad\div\vv\in\calG_{k-2}(\E)$ (as we do with our degrees of freedom \eqref{dof2} for
$\vv\in\Vfdk(\E)$), is equivalent to assign $\div\vv\in\P_{k-1}(\E)$ up to an additive constant. This constant
will be assigned by the integral of $\vv\cdot\nn$ over $\partial\E$, that can be deduced from the degrees of freedom \eqref{dof1}.
Indeed, using the same integration by parts applied in
\eqref{added:X1}, the degrees of freedom \eqref{dof1} and
\eqref{dof2} allow us to compute $\int_{\E}\div\vv\,q_{k-1}\dE$
for all $q_{k-1}\in\P_{k-1}(\E)$, and since $\div\vv\,\in\P_{k-1}(\E)$,
 we can compute exactly  the divergence of any $\vv\in\Vfdk(\E)$.
In turn this implies, again by using an integration by parts and
\eqref{dof1}, that we are able to compute also
$$
\int_{\E}{\vv\cdot \bg_{k}}\dE \quad \forall \bg_{k}\in\calG_{k} ,
$$
and actually
$$
\int_{\E}{\vv\cdot \bgrad \varphi}\dE \quad \forall \varphi \mbox{ polynomial on } \E .
$$

The above property, combined with \eqref{dof3}, allows to compute
the integrals against any $\qq_k \in (\P_{k}(\E))^2$ and thus
yields the following important result.

\begin{theorem}
The $L^2(\E)$ projection operator
$$
\Pi^0_{k} \: : \Vfdk(\E) \ \longrightarrow \ (\P_k(\E))^2
$$
is computable using the degrees of freedom \eqref{dof1}--\eqref{dof3}.
\end{theorem}

{\color{dblue} \begin{remark}We point out that, for instance, the $(L^2(\E))^2$ projection would be
much more difficult to compute if we used the degrees of freedom discussed in Remark \ref{otherdoff2}.\end{remark}}

\subsection{ The global 2D-face space}

Given a polygon $\Omega$ and a decomposition $\Th$ of $\Omega$ into a finite number of polygonal elements $\E$, we can now consider the {\it global} space
\begin{multline}\label{Vface2d-glo}
\Vfdk(\Omega):=\{\vv\in H(\div;\Omega)\cap H(\rot_h;\Omega)  \mbox{ s. t. } \vv\cdot\nn_{|e}\in\P_{k}(e)~\forall \mbox{ edge $e$ in } \Th,\\
\bgrad\div\vv\in\calG_{k-2}(\E), \mbox{ and } \rot\vv\in\P_{k-1}(E)~\forall \mbox{ element $\E$ in}
\Th\},
\end{multline}
where, with a common notation, $H(\rot_h;\Omega)$ is the space of vector valued functions $\vv$ in
$(L^2(\Omega))^2$ such that their $\rot$, {\it within each element} $\E$, belongs to $L^2(\E)$. In
other words
\begin{equation}\label{roth}
H(\rot_h;\Omega)=\displaystyle{\prod_{\E\in\Th} H(\rot;\E)}.
\end{equation}
Note that in \eqref{Vface2d-glo} we assumed that the elements $\vv$ of  $\Vfdk(\Omega)$ have a divergence that is {\it globally} (and not just element-wise)  in $L^2(\Omega)$. Hence  the normal component of vectors $\vv\in\Vfdk(\Omega)$ will have to be {``continuous''}
(with obvious meaning)  at the inter-element edges. From the local degrees of freedom \eqref{dof1}-\eqref{dof3} we deduce the global degrees of freedom:
\begin{eqnarray}
&\int_e{\vv\cdot\nn}\,{p}_{\,k}\,\de
&\quad \mbox{ for all edge $e$, for all }\;
p_{k}\in\P_k(e),\label{dof1glo}\\[3pt]
&\int_{\E}{\vv\cdot \bg_{k-2}}\dE &\quad\mbox{ for all element $\E$, for all  $\bg_{k-2}\in\calG_{k-2}(\E)$},  \label{dof2glo}
\\[3pt]
&\int_{\E}{\vv\cdot \bg_{k}^\perp}\dE  &\quad \mbox{ for all element $\E$, for all  $\bg_{k}^\perp\in\calG_{k}^\perp(\E)$}
.\label{dof3glo}
\end{eqnarray}
From the above discussion it follows immediately that the degrees of freedom \eqref{dof1glo}-\eqref{dof3glo} are unisolvent, and that the dimension of $\Vfdk(\Omega)$ is given by
\begin{equation*}
\begin{split}
\dim(\Vfdk(\Omega))=\ & \pi_{k,1}\times\{\text{number of edges in }\Th\} +\\ & (2\pi_{k-1,2}-1)\times\{\text{number of elements in }\Th\}.
\end{split}
\end{equation*}
%

%
%
%
%
%
%
%
%
%
%
%
%
%
%
%
%

\section{2D Edge Elements}\label{sec:edge:2D} 

The edge elements in 2D exactly correspond to the face elements, just rotating everything by $\pi/2$.
For the sake of completeness we just recall the definition of the spaces and the
corresponding degrees of freedom.

\subsection {The local space} 

On a polygon $\E$ we set
\begin{multline}\label{L:2Dedgespace}
\Vedk(\E):=\{\vv\in H(\div;\E)\cap H(\rot;\E):
\vv\cdot\tt_{|e}\in\P_{k}(e)\,\forall \mbox{ edge $e$ of } \E,\\
\brot\rot\vv\in\P_{k-2}(E), \mbox{ and }
\div\vv\in\P_{k-1}(\E) \}.
\end{multline}

\subsection{The Degrees of Freedom} 

{\color{dblue} A convenient set of degrees of freedom for elements  $\vv$  in $\Vedk(\E)$ will be:}
\begin{eqnarray}
&\int_e{\vv\cdot\tt}\,p_{k}\,\de
&\quad \mbox{ for all edge $e$, for all }\;
p_{k}\in\P_k(e),\label{doe1}\\[3pt]
&\int_{\E}{\vv\cdot\rr_{k-2}}\dE  &\quad \mbox{ for all $\rr_{k-2}\in\calR_{k-2}$}
,\label{doe3}
\\[3pt]
&\int_{\E}{\vv\cdot \rr^{\perp}_{k}}\dE &\quad \mbox{ for all } \rr^{\perp}_{k}\in\calR^{\perp}_{k}.  \label{doe2r}
\end{eqnarray}

\begin{remark}\label{otherdofe2} Here too we could use alternative degrees of freedom,
in analogy with the ones discussed in Remarks \ref{alternativedof}  and \ref{otherdoff2}.
In particular we point out that we can identify uniquely an element $\vv$ of $\Vedk(\E)$
by prescribing its tangential component
$\vv\cdot\tt$ (in $\P_k(e)$) on every edge, its rotation $\rot\vv$ (in $(\P_{k-1}(\E))/{\R}$), and its divergence
$\div\vv$ (in $\P_{k-1}(\E)$).
\end{remark}
\begin{remark} Obviously, here too we can define the $L^2-$projection onto $\P_k$, exactly as we did in subsection \ref{computing:face:2d}, with $\calR_k^\perp$ taking the role of $\calG_k^\perp$.
\end{remark}

\subsection{ The global 2D-edge space } 

Given a polygon $\Omega$ and a decomposition $\Th$ of $\Omega$ into a finite number of polygonal elements $\E$, we can now consider the {\it global} space
\begin{multline}\label{Vedge2d-glo}
\Vedk(\Omega):=\{\vv\in H(\div_h;\Omega)\cap H(\rot;\E)  \mbox{ s. t. } \vv\cdot\tt_{|e}\in\P_{k}(e)~\forall \mbox{ edge $e$ in } \Th,\\
\div\vv\in\P_{k-1}(\E), \mbox{ and } \brot\rot\vv\in\calR_{k-2}(\E)~\forall \mbox{ element $\E$ in}
\Th\},
\end{multline}
where, with a notation similar to that used in \eqref{roth}, we have here
\begin{equation}\label{divh} H(\div_h;\Omega)=\displaystyle{\prod_{\E\in\Th} H(\div;\E)}.
\end{equation}

Note that the tangential component of vectors $\vv\in\Vedk(\Omega)$ will have to be {``continuous''}
(with obvious meaning)  at the inter-element edges. From the local degrees of freedom \eqref{doe1}-\eqref{doe2r} we deduce the global degrees of freedom:
\begin{eqnarray}
&\int_e{\vv\cdot\tt}\,{p}_{\,k}\,\de
&\quad \mbox{ for all edge $e$, for all }\;
p_{k}\in\P_k(e),\label{doe1glo}\\[3pt]
&\int_{\E}{\vv\cdot \rr_{k-2}}\dE &\quad\mbox{ for all element $\E$, for all  $\rr_{k-2}\in\calR_{k-2}(\E)$},  \label{doe3glo}
\\[3pt]
&\int_{\E}{\vv\cdot \rr_{k}^\perp}\dE  &\quad \mbox{ for all element $\E$, for all  $\rr_{k}^\perp\in\calR_{k}^\perp(\E)$}
.\label{doe2glor}
\end{eqnarray}
From the above discussion it follows immediately that the degrees of freedom \eqref{doe1glo}-\eqref{doe2glor} are unisolvent, and that the dimension of $\Vedk(\Omega)$ is
\begin{equation*}
\begin{split}
\dim(\Vedk(\Omega))=\ &\pi_{k,1}\times\{\text{number of edges in }\Th\} + \\ &(2\pi_{k-1,2}-1)\times\{\text{number of elements in }\Th\}.
\end{split}
\end{equation*}
%
%

%
%
%
%
%
%
%
%
%
%
%
%
%
%
%
%

\section{3D Face Elements} 

The three-dimensional $H(\div)$-conforming spaces follow in a very natural
way the path of their two-dimensional companions.

\subsection{The local space.} 

On a polyhedron $\PP$ we set
\begin{multline}\label{def:face:space:3D}
\Vftk(\PP):=\{\vv\in H(\div;\PP)\cap H(\bcurl;\PP)  \mbox{ s. t. } \vv\cdot\nn_{\PP}^f\in\P_k(f)\,\forall \mbox{ face $f$ of } \PP,\\
\bgrad\div\vv\in\calG_{k-2}(\PP),
~\bcurl\vv\in\calR_{k-1}(\PP)\}.
\end{multline}

\subsection{Dimension of the space $\Vftk(\PP)$}   

{\color{dgreen}We recall from the introduction that given}
 \begin{itemize}
\item a function $g$ defined on $\partial\PP$
such that $g_{|f}\in\P_{k}(f)$ for all $f \in \partial E$,
\item a polynomial $f_d\in\P_{k-1}(\PP)$ such that
\begin{equation}\label{compat3d}
\int_{\PP}f_d\dP=\int_{\partial\PP} g \dS ,
\end{equation}
\item a vector valued polynomial $\bbf_r\in\calR_{k-1}(\PP)$ ,
\end{itemize}
we can find a unique vector $\vv\in\Vftk(\PP)$ such that
\begin{equation}\label{perdimftd}
\vv\cdot\nn=g \mbox{ on }\partial\PP, \quad \div\vv=f_d\mbox{ in }\PP, \quad \bcurl\vv=\bbf_r\mbox{ in }\PP.
\end{equation}
 This easily implies that the dimension of $\Vftk(\PP)$ is given by:
{\it the number of faces, $\ell_f$, times the dimension of $\P_{k}$ (in $\R^2$)}, plus {\it the dimension of
$\P_{k-1}(\PP)$ minus one (to take into account the compatibility condition \eqref{compat3d})}
 plus {\it the dimension of  $\calR_{k-1}(\PP)$}, that is
\begin{equation}\label{dimftd}
\dim(\Vftk(\PP))=\ell_f\pi_{k,2}+\gamma_{k-2,3}+\rho_{k-1,3}.
\end{equation}

\subsection{The Degrees of Freedom}   

The degrees of freedom  will be:
\begin{eqnarray}
&\int_f{\vv\cdot\nn_{\PP}^f}\,p_{k}\,\df
&\quad \mbox{ for all face $f$, for all }\;
p_{k}\in\P_k(f),\label{dof31}\\[3pt]
&\int_{\PP}{\vv\cdot\bg_{k-2}}\dP &\quad\mbox{ for all $\bg_{k-2}\in\calG_{k-2}$},  \label{dof32d}
\\[3pt]
&\int_{\PP}{\vv\cdot\bg^{\perp}_{k}}\dP  &\quad \mbox{ for all $\bg^{\perp}_{k}\in\calG^{\perp}_{k}$}
.\label{dof33}
\end{eqnarray}
\color{black}
It is not difficult to check, using \eqref{dimG} and \eqref{dimGperp3}, that the number of the above
degrees of freedom is given by
\begin{equation}
\ell_f\pi_{k,2}+\dim\{\calG_{k-2}\}+\dim\{\calG^{\perp}_{k}\}=\ell_f\pi_{k,2}+\gamma_{k-2,3}+\rho_{k-1,3},
\end{equation}
which equals the dimension of $\Vftk(\PP)$ as given in \eqref{dimftd}.

\subsection{Unisolvence}  

Having already noticed that the number of degrees of freedom
\eqref{dof31}-\eqref{dof33} equals the dimension of $\Vftk(\PP)$, we just have to show that
 if a $\vv\in\Vftk(\PP)$ verifies
\begin{eqnarray}
&\int_f{\vv\cdot\nn_{\PP}^f}\,p_{k}\,\df=0
&\quad \mbox{ for all face $f$, for all }\;
p_{k}\in\P_k(f),\label{dof310}\\[3pt]
&\int_{\PP}{\vv\cdot\bg_{k-2}}\dP=0 &\quad\mbox{ for all $\bg_{k-2}\in\calG_{k-2}$},  \label{dof320}
\\[3pt]
&\int_{\PP}{\vv\cdot\bg^{\perp}_{k}}\dP=0  &\quad \mbox{ for all $\bg^{\perp}_{k}\in\calG^{\perp}_{k}$},
\label{dof330}
\end{eqnarray}
then $\vv=0$. We proceed as in the two dimensional case. For this  we observe first that
if $\vv\in \Vftk(\PP)$ and if \eqref{dof310} and \eqref{dof320} hold, then
\begin{equation}\label{orthograd3}
\int_{\PP} \vv\cdot\bgrad\varphi\dP=0\quad
\forall\varphi\in H^1(\PP).
\end{equation}
The proof is identical to that of Lemma \ref{primi0}. Then we observe that for all  $\vv\in \Vftk(\PP)$ there exist a $\qq_{k}^\perp$ in $\calG_k^\perp$  and a $\varphi\in H^1(\PP)$
such that
\begin{equation}\label{anche3d}
\vv=\qq_k^\perp+\grad\varphi .
\end{equation}
Again the proof is identical to that of Lemma \ref{deco1}, this time using [\ref{ex3};iv)] to show the existence of
a $\qq_{k}^\perp\in\calG_k^\perp$ such that $\bcurl(\vv-\qq_k^{\perp})=0$. Then using \eqref{anche3d} we conclude that $\vv=0$ as in \eqref{mungi}.
\begin{remark}\label{alternativedof3f} As we did in the $2D$ case, we point out that the degrees
of freedom \eqref{dof32d} or \eqref{dof33} could be replaced by equivalent ones. In particular, the degrees of freedom \eqref{dof32d} can again be replaced by
\begin{equation}
\int_{\PP}\div\vv\,q_{k-1}\dP \qquad \mbox{for all } q_{k-1}\in\P_{k-1}/\R,
\end{equation}
and the degrees of freedom \eqref{dof33} could be substituted by
\begin{equation}
\int_{\PP}\bcurl\vv\,\qq_{k-1}\dP \qquad \mbox{for all } \qq_{k-1}\in\calR_{k-1}.
\end{equation}
\end{remark}
{\color{dred}
\begin{remark} Obviously, here too we can {\color{dblue} compute} the $L^2-$projection onto $\P_k$, exactly as we did in subsection \ref{computing:face:2d}.
\end{remark}
}
\begin{remark}\label{otherdoff3} In the same spirit of Remark \ref{otherdoff2}, we point out that we can identify uniquely an element $\vv$ of $\Vftk(\E)$
by prescribing its normal component
$\vv\cdot\nn$ (in $\P_k(f)$) on each face, its rotation $\bcurl\vv$ (in $\calR_{k-1}(\E)$), and its divergence
$\div\vv$ (in $(\P_{k-1}(\E))/{\R}$).
\end{remark}

{\color{black}
\subsection{The global 3D-face space}  

Having now a polyhedron $\Omega$ and a decomposition $\Th$  of $\Omega$ into a finite number of
polyhedral elements $\PP$, we can  consider the global space:{\color{dblue}
\begin{equation}\label{def:face:space:3D-glo}
\begin{aligned}
\Vftk(\Omega):= \{ & \vv\in H(\div;\Omega)\cap H(\bcurl_h;\Omega)  \mbox{ such that: }\\ & \vv\cdot\nn_{\PP}^f\in\P_k(f)\,\forall \mbox{ face $f$ in } \Th,
\;\bgrad\div\vv\in\calG_{k-2}(\PP),\\
& \mbox{ and }
~\bcurl\vv\in\calR_{k-1}(\PP)\,\forall \mbox{ element $\PP$ in } \Th \},
\end{aligned}
\end{equation}
\noindent with obvious notation (in agreement with \eqref{roth} and \eqref{divh})
for the operator $\bcurl_h$ and the corresponding space $H(\bcurl_h;\Omega)$.
As we did for the 2D case, we note that the normal component of the elements of $\Vftk(\Omega)$ will be
``continuous'' at the inter-element face. In $\Vftk$ we can take, as degrees of freedom:
\begin{eqnarray}
&\int_f{\vv\cdot\nn_{\PP}^f}\,p_{k}\,\df
&\quad \mbox{ for all face $f\in\Th$, for all }\;
p_{k}\in\P_k(f),\label{dof31glo}\\[3pt]
&\int_{\PP}{\vv\cdot\bg_{k-2}}\dP &\quad\mbox{ for all element $\PP\in\Th$, for all $\bg_{k-2}\in\calG_{k-2}(\PP)$}, \quad  \label{dof32dglo}
\\[3pt]
&\int_{\PP}{\vv\cdot\bg^{\perp}_{k}}\dP  &\quad \mbox{ for all element $\PP\in\Th$, for all $\bg^{\perp}_{k}\in\calG^{\perp}_{k}(\PP)$}
.\label{dof33glo}
\end{eqnarray}
From the above discussion it follows immediately that the degrees of freedom \eqref{dof31glo}-\eqref{dof33glo} are unisolvent, and that the dimension of $\Vftk(\Omega)$ is
\begin{equation*}
\begin{split}
\dim(\Vftk(\Omega))= &\ \pi_{k,2}\times\{\text{number of faces in }\Th\} + \\ &(\pi_{k-1,3}-1+\rho_{k-1,3})\times\{\text{number of elements in }\Th\}.
\end{split}
\end{equation*}
%
%
%
}

%
%
%
%
%
%
%
%
%
%
%
%
%
%
%
%

\section{3D Edge Elements}  

This time we cannot just rotate the $3D$-face case. However we can get some inspiration.
We recall, from the very beginning, the Green formula:
\begin{equation}\label{greenrot}
\int_{\PP}\bcurl\bpsi\cdot\bphi\dP=\int_{\PP}\bpsi\cdot\bcurl\bphi\dP
+\int_{\partial\PP}\bpsi\cdot(\bphi\wedge\nn)\,{\rm d}S,
\end{equation}
as well as
\begin{equation}\label{greenrotrot}
\int_{\PP}\bcurl\bpsi\cdot\bcurl\bphi\dP=\int_{\PP}\bpsi\cdot\Big[-\Delta\bphi
+\bgrad\div\bphi\Big]\dP+\int_{\partial\PP}\bpsi\cdot(\bcurl\bphi\wedge\nn)\,{\rm d}S.
\end{equation}
We also recall the observation that we made in Section \ref{nota}  concerning the
difference between  $\bphi\wedge\nn_f$ and   $\bphi_f$.
We introduce moreover the following space.

\begin{definition}\label{smooth-bound-space}
We define the boundary space $\calB(\partial\PP)$ as
the space of $\vv$ {\color{dblue}in} $(L^2(\partial\PP))^3$ such that $\vv_{f}\in {\color{dblue} H(\div;f)\cap H(\rot;f)}${ on each face }
$f\in\partial\PP$, and such that on each edge $e$ (common to the faces $f_1$ and $f_2$),  $\vv_{f_1}\cdot{\bf t}_e$ and $\vv_{f_2}\cdot{\bf t}_e$ (where ${\bf t}_e$ is a unit tangential vector to $e\,$) coincide . Then we define $\calB_t(\partial\PP)$ as the space
of the tangential components of the elements of $\calB(\partial\PP)$.
\end{definition}

\begin{definition}\label{def:boundspace}
We now define the boundary VEM space $B_k^{\rm edge}(\partial\PP)$ as
$$
B_k^{\rm edge}(\partial\PP) = \big\{ \vv \in \calB_t(\partial\PP) \textrm{ such that } \vv_{f}\in\Vedk(f) \textrm{ on each face }
f\in\partial\PP \big\}.
$$
\end{definition}

Recalling the previous discussion on the two-dimensional virtual elements $\Vedk(f)$, we can easily see
that for a polyhedron with $\ell_e$ edges and $\ell_f$ faces the dimension
$\beta_k$ of $B_k^{\rm edge}(\partial\PP)$ is given by
\begin{equation}\label{defbetk}
\beta_k=\ell_e\pi_{k,1}+\ell_f~(2\pi_{k-1,2}-1).
\end{equation}

\subsection{ The local space}   

On a polyhedron $\PP$ we set{\color{dblue}
\begin{multline}\label{defVetkP}
\Vetk(\PP):=\{\vv|~\vv_{t}\in B_k^{\rm edge}(\partial\PP),
\div\vv\in\P_{k-1}(\PP), \mbox{ and } \bcurl\bcurl\vv\in\calR_{k-2}(\PP)\}.
\end{multline}}

\subsection{Dimension of the space  $\Vetk(\PP)$}  


We start by {\color{dgreen} observing } that,  given a vector $\bg$ in $B_k^{\rm edge}(\partial\PP)$, a function $f_d$ in $\P_{k-1}$, and
a vector ${\bbf}_r\in\calR_{k-2}(\PP)$ we can find a unique $\vv$ in $\Vetk(\PP)$ such that
\begin{equation}\label{esiedge}
\mbox{$\vv_{\tg}=\bg$ on $\partial\PP$, $\div\vv=f_d$ in $\PP$, and $\bcurl\bcurl\vv={\bbf}_r$ in $\PP$}.
\end{equation}

{\color{dgreen}
To prove it  we consider the following  auxiliary problems.
The first is: find $\HH$ in $(H^1(\PP))^3$ such that
\begin{equation}\label{defH}
\mbox{  $\bcurl\HH=\bbf_r$ in $\PP$,
$\div\HH=0$ in $\PP$, and $\HH\cdot\nn=\rot_2\bg$ on $\partial\PP$},
\end{equation}
that is uniquely solvable since
\begin{equation}
\int_{\partial\PP}\rot_2\bg\dS=0.
\end{equation}
The second is: find $\bpsi$ in $(H^1(\PP))^3$ such that
\begin{equation}\label{defpsi}
\mbox{  $\bcurl\bpsi=\HH$ in $\PP$,
$\div\bpsi=0$ in $\PP$, and $\bpsi_t=\bg$ on $\partial\PP$},
\end{equation}
that is also uniquely solvable since
\begin{equation}
\HH\cdot\nn=\rot_2\bg .
\end{equation}
The third problem is: find $\varphi\in H^1_0(\PP)$ such that:
\begin{equation}
\mbox{ $\Delta\varphi=f_d$ in $\PP$},
\end{equation}
that also has a unique solution.} Then it is not difficult to see that the choice
\begin{equation}\label{sole3d}
\vv:=\bpsi+\bgrad\varphi
\end{equation}
solves our problem. Indeed, it is clear that $(\bgrad \bphi)_{\tg}=0$, that $\div (\bgrad \bphi)=f_d$ and that $\bcurl\bcurl(\bgrad \bphi)=0$; all these, added to \eqref{defH} and \eqref{defpsi}, produce the right conditions. It is also clear
that the solution $\vv$ of \eqref{esiedge} is unique.

Hence we can conclude that the dimension of $\Vetk(\PP)$ is given by
\begin{equation}\label{dimVetk}
\dim(\Vetk(\PP))=\beta_k+\pi_{k-1,3}+\rho_{k-2,3} .
\end{equation}

\subsection{The Degrees of Freedom.}   
A possible set of degrees of freedom will be:
\begin{itemize}
\item for every edge $e$:
\end{itemize}
\vskip-0.8truecm
\begin{equation}\int_e{\vv\cdot\tt}\,p_{k}\,\de
\quad \mbox{ for all }\;
p_{k}\in\P_k(e),\label{doe13}
\end{equation}
\begin{itemize}
\item for every face $f$:
\end{itemize}
\vskip-0.8truecm
\begin{eqnarray}
&\int_{f}{\vv\cdot \rr^{\perp}_{k}}{\color{dblue}\df }\quad \mbox{ for all } \rr^{\perp}_{k}\in\calR^{\perp}_{k}(f),  \label{doe23r}
\\[3pt]
&\int_{f}{\vv\cdot\rr_{k-2}}{\color{dblue}\df } \quad \mbox{ for all $\rr_{k-2}\in\calR_{k-2}(f)$},
\label{doe33e}
\end{eqnarray}
\begin{itemize}
\item and inside $\PP$
\end{itemize}
\vskip-0.8truecm
\begin{eqnarray}
&\int_{\PP}{\vv\cdot\rr^{\perp}_{k}}\dP  \quad \mbox{ for all }\rr^{\perp}_{k}\in
\calR^{\perp}_{k},
\label{doe34ec}
\\[3pt]
&\int_{\PP}{\vv\cdot\rr_{k-2}}\dP  \quad \mbox{ for all $\rr_{k-2}\in\calR_{k-2}$}.
\label{doe35e}
\end{eqnarray}
The total number of degrees of freedom \eqref{doe13}-\eqref{doe33e} is clearly equal to $\beta_k$ {\color{dblue} as given in \eqref{defbetk}}
and the number of degrees of freedom \eqref{doe35e} is equal to $\rho_{k-2,3}$. On the other hand, using [\ref{ex3};v)]
we see that the number of degrees of freedom \eqref{doe34ec} is equal to $\pi_{k-1,3}$, so that the total number
of degrees of freedom \eqref{doe13}-\eqref{doe35e} is equal to the dimension of $\Vetk(\PP)$ as computed in \eqref{dimVetk}.

\subsection{Unisolvence.}  
Having seen that the number of degrees of freedom \eqref{doe13}-\eqref{doe35e} equals the dimension of
$\Vetk(\PP)$, in order to see their unisolvence we only need to check that a vector $\vv\in\Vetk(\PP)$
that satisfies
\begin{eqnarray}
&\int_e{\vv\cdot\tt}\,p_{k}\,\de=0
&\quad {\color{dblue}\forall}\mbox{ edge $e$ of $\PP$ and ${\color{dblue}\forall}$}\,
p_{k}\in\P_k(e),\label{doe130}\\[3pt]
&\int_{f}{\vv\cdot \rr^{\perp}_{k}}{\color{dblue}\df}=0 &\quad {\color{dblue}\forall}\mbox{ face $f$ of $\PP$ and ${\color{dblue}\forall}$  $\rr^{\perp}_{k}\in\calR^{\perp}_{k}(f)$},  \label{doe23r0}
\\[3pt]
&\int_{f}{\vv\cdot\rr_{k-2}}{\color{dblue}\df}=0  &\quad {\color{dblue}\forall}\mbox{ face $f$ of $\PP$ and ${\color{dblue}\forall}$  $\rr_{k-2}\in\calR_{k-2}(f)$},
\label{doe33e0}\\[3pt]
&\int_{\PP}{\vv\cdot\rr^{\perp}_{k}}\dP =0 &\quad {\color{dblue}\forall}\mbox{ }\rr^{\perp}_{k}\in
\calR^{\perp}_{k}(\PP),
\label{doe34e0}
\\[3pt]
&\int_{\PP}{\vv\cdot\rr_{k-2}}\dP=0 &\quad {\color{dblue}\forall}\mbox{ $\rr_{k-2}\in\calR_{k-2}(\PP)$},
\label{doe35e0}
\end{eqnarray}
is necessarily equal to zero.

Actually,  recalling the results of Section \ref{sec:edge:2D}, it is pretty obvious that \eqref{doe130}-\eqref{doe33e0} imply that $\vv_{\tg}=0$ on
$\partial\PP$. Moreover, since $\bcurl\bcurl\vv \in \calR_{k-2}(\PP)$, we are allowed to take
$\rr_{k-2} = \bcurl\bcurl\vv$ as a test function in \eqref{doe35e0}.
An integration by parts (using $\vv_{\tg}=0$) gives
\begin{equation}\label{L:int_curl}
0 = \int_\PP \vv \cdot \bcurl\bcurl\vv \dP
= \int_\PP (\bcurl\vv)\cdot (\bcurl\vv) \dP
\end{equation}
and therefore we get $\bcurl\vv=0$. Using this, and again $\vv_{\tg}=0$, we easily check, integrating
by parts, that
\begin{equation}\label{ortaicurl}
\int_{\PP}\vv\cdot\bcurl\bphi\dP=0\quad\forall\bphi\in (H^1(\PP))^3.
\end{equation}
Now we recall that from the definition \eqref{defVetkP} of $\Vetk(\PP)$ we have that $\div\vv$ is in $\P_{k-1}$. From
[(\ref{ex3});v] we then deduce that there exists a $\qq_k^{\perp}\in\calR_k^\perp$ with $\div\qq_k^{\perp}=\div\vv$,
so that the divergence of $\vv-\qq_k^{\perp}$ is zero, and then (since $\PP$ is simply connected)
\begin{equation}\label{isacurl}
\vv-\qq_k^{\perp}=\bcurl\bphi
\end{equation}
for some {\color{dblue}$\bphi\in H(\bcurl;\PP)$}. At this point we can use \eqref{ortaicurl} and \eqref{isacurl} to conclude as in
\eqref{mungi}
{\color{dblue}\begin{multline}
\int_{\PP}|\vv|^2\dP=\int_{\PP}\vv\cdot(\qq_k^\perp+\bcurl\bphi)\dP=\int_{\PP}\vv\cdot\qq_k^\perp\dP+
\int_{\PP}\vv\cdot\bcurl\bphi \dP=0.
\end{multline}}

\subsection{Alternative degrees of freedom}  
 As we did in the previous cases, we observe that the degrees of freedom \eqref{doe13}-\eqref{doe35e}
are not (by far) the only possible choice. To start with, we can change the degrees of freedom
in each face, according to Remark \ref{alternativedof}. Moreover, in the spirit of \eqref{esiedge} we
could assign, instead of \eqref{doe34ec} and/or \eqref{doe35e},  $\bcurl\bcurl\vv$ in $\calR_{k-2}(\PP)$
and/or $\div\vv$ in $\P_{k-1}(\PP)$, respectively.

\subsection{The global 3D-edge space} Here too we can assume that we have a polyhedral
domain $\Omega$ and its decomposition $\Th$ in a finite number of polyhedra $\PP$. In this case we can define the global space
{\color{dblue}
\begin{multline}\label{defVetkPglo}
\Vetk(\Omega):=\{\vv\in H(\div_h;\Omega)\cap H(\bcurl;\Omega)\mbox{ s. t. }
\forall  \PP\in\Th
 \mbox{ we have: }\\~\vv_{t}\in B_k^{\rm edge}(\partial\PP),\,
\div\vv\in\P_{k-1}(\PP), \mbox{ and } \bcurl\bcurl\vv\in\calR_{k-2}(\PP)\}.
\end{multline}}
Accordingly, we could take, as degrees of freedom:
\begin{itemize}
\item for every edge $e$ in $\Th$:
\end{itemize}
\begin{equation}\int_e{\vv\cdot\tt}\,p_{k}\,\de
\quad \mbox{ for all }\;
p_{k}\in\P_k(e),\label{doe13glo}
\end{equation}
\begin{itemize}
\item for every face $f$ in $\Th$:
\end{itemize}
\begin{eqnarray}
&\int_{f}{\vv_f\cdot \rr^{\perp}_{k}}\df \quad \mbox{ for all } \rr^{\perp}_{k}\in\calR^{\perp}_{k}(f),  \label{doe23rglo}
\\[3pt]
&\int_{f}{\vv_f\cdot\rr_{k-2}}\df  \quad \mbox{ for all $\rr_{k-2}\in\calR_{k-2}(f)$}
\label{doe33eglo}
\end{eqnarray}
\begin{itemize}
\item and for every element $\PP$ in $\Th$
\end{itemize}
\begin{eqnarray}
&\int_{\PP}{\vv\cdot\rr^{\perp}_{k}}\dP  \quad \mbox{ for all }\rr^{\perp}_{k}\in
\calR^{\perp}_{k},
\label{doe34ecglo}
\\[3pt]
&\int_{\PP}{\vv\cdot\rr_{k-2}}\dP  \quad \mbox{ for all $\rr_{k-2}\in\calR_{k-2}$}.
\label{doe35eglo}
\end{eqnarray}
From the above discussion it follows immediately that the degrees of freedom \eqref{doe13glo}-\eqref{doe35eglo} are unisolvent, and that the dimension of $\Vetk(\Omega)$ is
\begin{equation}
\begin{aligned}
\dim(\Vetk(\Omega))= & \pi_{k,1}\times\{\text{number
of edges in } \Th\} \\
& + (2\pi_{k-1,2}-1)\times \{\text{number of faces in } \Th\}\\
& +(\pi_{k-1,3}+\rho_{k-1,3})\times \{\text{number of elements in } \Th\}.
\end{aligned}
\end{equation}


\subsection{An enhanced edge space}  

It is immediate to check that the degrees of freedom \eqref{doe34ec}-\eqref{doe35e}  allow to compute the moments of $\vv\in\Vetk(\PP)$ up to order $k-2$.
Nevertheless, in order to be able to compute the $L^2(\PP)$ projection operator on the space
$(\P_{k}(\PP))^3$ we need to be able to compute the moments up to order $k$.
In the present section, in the spirit of \cite{projectors}, we will introduce an enhanced space $\Wetk(\PP)$ with the additional property that the $L^2$ projector on $(\P_{k}(\PP))^3$ is computable.

We consider the larger virtual space
{\color{dblue}
\begin{multline}
\Vettk(\PP) :=
\{\vv|~\vv_{|\partial\PP}\in B_k^{\rm edge}(\partial\PP),
\div\vv\in \P_{k-1}(\PP),
\mbox{ and }
\,\bcurl\bcurl\vv\in\calR_{k}(\PP)\}.
\end{multline}}
Following the same identical arguments used in the previous section and introducing the space
$$
\Rort(\PP) :=\big\{ \qq_k \in \calR_{k} \: : \: \int_\PP \qq_k \cdot \rr_{k-2} \dP = 0 \
\forall \rr_{k-2} \in \calR_{k-2} \big\},
$$
it is immediate to check that \eqref{doe13}-\eqref{doe35e}, with the addition of
\begin{equation}\label{doe36}
\int_{\PP}{\vv\cdot \qq_{k}}\dP  \quad \mbox{ for all $\qq_{k}
\in \Rort(\PP)$} ,
\end{equation}
constitute a set of degrees of freedom for $\Vettk(\PP)$. Note moreover that $\Vetk(\PP)$ {\color{dblue}is a subset of} $\Vettk(\PP)$ and that the combination of \eqref{doe34ec}, \eqref{doe35e} and \eqref{doe36} allows, for any function in $\Vettk(\PP)$, to compute all the integrals against polynomials in $\P_k(\PP)$. Therefore the $L^2$ projection operator
$$
\Pi^0_k \: : \: \Vettk(\PP) \rightarrow \Big( \P_{k}(\PP) \Big)^3
$$
is computable.

For the time being we \emph{assume} the existence of a projection operator
\begin{equation}\label{L:Pit}
\Pit_k \: : \: \Vettk(\PP) \rightarrow \Big( \P_{k}(\PP) \Big)^3 ,
\end{equation}
with the fundamental property of depending \emph{only} on the degrees of freedom \eqref{doe13}-\eqref{doe35e} (meaning that if $\vv$ satisfies \eqref{doe130}-\eqref{doe35e0}
then $\Pit_k\vv=0$).
We now introduce the space
\begin{multline}\label{L:Wdef}
\Wetk(\PP) :=\{\vv \in \Vettk(\PP) \mbox{ such that: } \\
\int_\PP (\Pit_k \vv) \cdot \qq_k \dP = \int_\PP (\Pi^0_k \vv) \cdot \qq_k \dP
\quad \forall \qq_{k} \in \Rort(\PP) \}.
\end{multline}
We then have the following lemma.

\begin{lemma}
The dimension of the space $\Wetk(\PP)$ is equal to the dimension of the original edge space $\Vetk(\PP)$. Moreover, the operators \eqref{doe13}-\eqref{doe35e} constitute a set of degrees of freedom for $\Wetk(\PP)$.
\end{lemma}
\begin{proof}
By definition of $\Wetk(\PP)$ we have
$$
\dim \Big( \Wetk(\PP) \Big) \ge \dim \Big( \Vettk(\PP) \Big) -  \dim \Big( \Rort(\PP) \Big) = \dim \Big( \Vetk(\PP) \Big) .
$$
Therefore, in order to conclude the lemma, it is sufficient to show the unisolvence of
\eqref{doe13}-\eqref{doe35e}. For this, let $\vv \in \Wetk(\PP)$ satisfying \eqref{doe130}-\eqref{doe35e0}. Note that, by the previously mentioned property of the (linear) projection operator $\Pit_k$, we immediately have  that $\Pit_k(\vv)$ is equal to $0$. Therefore, by definition of $\Wetk(\PP)$, for all $\qq_{k} \in \Rort(\PP)$ it holds
\begin{equation}\label{extra-dof0}
\int_{\PP}{\vv\cdot \qq_{k}}\dP = \int_{\PP}{ \Big(\Pi^0_k \vv \Big) \cdot \qq_{k}}\dP
=\int_{\PP}{\Big(\Pit_k \vv \Big) \cdot \qq_{k}}\dP = 0 .
\end{equation}
Since $ \Wetk(\PP) \subseteq \Vettk(\PP)$ and the set of degrees of freedom \eqref{doe13}-\eqref{doe35e} plus \eqref{doe36} is unisolvent for $\Vettk(\PP)$, we conclude that \eqref{doe130}-\eqref{doe35e0} plus \eqref{extra-dof0} imply $\vv=0$.
\end{proof}

Note that, due to the above lemma, the enhanced space $\Wetk(\PP)$ has the same degrees of freedom as $\Vetk(\PP)$. Moreover, since the condition in \eqref{L:Wdef} is satisfied by polynomials of degree $k$, we still have $(\P_k(\PP))^3 \subseteq \Wetk(\PP)$.
The advantage of the space $\Wetk(\PP)$ with respect to $\Vetk(\PP)$ is that in $\Wetk(\PP)$ we can compute all the moments of order up to $k$. Indeed, the moments
$$
\begin{aligned}
& \int_{\PP}{\vv\cdot \qq_{k-2}}\dP  \quad \mbox{ for all $\qq_{k-2}\in\calR_{k-2}(\PP)$} , \\
& \int_{\PP}{\vv\cdot \qq_{k}^\perp}\dP  \quad \mbox{ for all $\qq_{k}\in\calR_{k}^{\perp}(\PP)$}
\end{aligned}
$$
can be computed using the degrees of freedom \eqref{doe34ec} and \eqref{doe35e}, while
\begin{equation}\label{projort}
\int_{\PP}{\vv\cdot \qq_{k}}\dP = \int_{\PP}{ \Big(\Pi^0_k \vv \Big) \cdot \qq_{k}}\dP
=\int_{\PP}{\Big(\Pit_k \vv \Big) \cdot \qq_{k}}\dP
\end{equation}
for all $\qq_k \in \Rort(\PP)$.

We are therefore left with the duty to build a projection operator $\Pit_k$ as in \eqref{L:Pit}.
Let
$N$ denote the dimension of the space $\Vetk(\PP)$, i.e. the number of degrees of freedom \eqref{doe13}-\eqref{doe35e}.  Let us introduce the operator
$$
\dofop : \Vettk(\PP) \longrightarrow {\mathbb R}^N
$$
that associates, to any $\vv\in\Vettk(\PP)$, a vector with components given by the evaluation of all the (ordered) operators \eqref{doe13}-\eqref{doe35e} on $\vv$ (in other words, $\dofop$ associates to every
element of $\Vettk(\PP)$ its ``first $N$''  degrees of freedom).
Note that the operator $\dofop$ is not injective (as the dimension of  $\Vettk(\PP)$ is {\it bigger}
than that of $\Vetk$, that in turn is equal to $N$). On the other hand, since $(\P_k(\PP))^3 \subseteq \Vetk(\PP)$ and the above $N$ operators are a set of degrees of freedom for $\Vetk(\PP)$,  the operator $\dofop$ {\it restricted to $(\P_k(\PP))^3$ is injective}.
{\color{dblue} Given now {\it any} symmetric and positive definite bilinear form $\calS$ defined on $\R^N\times\R^N$ we define the projection operator $\Pit_k^{\calS}$ as follows.
For all $\vv \in \Vettk(\PP)$:
\begin{equation}\label{proje}
\left\{
\begin{aligned}
&  \Pit_k^{\calS} \vv \in \big( \P_k(\PP) \big)^3 \\
& \calS\Big( \dofop \, \Pit_k^{\calS} \vv - \dofop \vv , \dofop {\bf q}_k\Big) = 0 \qquad \forall {\bf q}_k \in \big( \P_k(\PP) \big)^3.
\end{aligned}
\right.
\end{equation}}
 By recalling that $\dofop$ is injective on $(\P_k(\PP))^3$, it is immediate to check that the above operator is well defined. Moreover, by definition it depends only on the degrees of freedom \eqref{doe13}-\eqref{doe35e}.

{\color{dblue}
 \begin{remark}Our construction is pretty general. Actually it is not difficult to prove that
 for every projector $\calP$ onto $(\P_k(\PP))^3$ depending only on the degrees of freedom \eqref{doe13}-\eqref{doe35e} we can find a bilinear symmetric positive definite form $\calS$
 such that $\calP = \Pi_{k}^{\calS}$.
 \end{remark}

 \begin{remark} The construction of the enhanced space $\Wetk(\PP)$ has basically a theoretical interest. In practice (meaning, in writing the code) one doesn't even need to know what this space is. If one needs to use the $L^2$ projection of the elements of $\Vetk$, one can just
 use the construction \eqref{proje} (typically, with $\calS$ equal to the Euclidean scalar product
 in $\R^N$) in order to define $\Pit_k$, and then \eqref{projort} to get the $L^2$ projection.
 \end{remark}}

\section{Scalar VEM spaces}

In the present section we restrict our {reminders} to the three dimensional case, the two dimensional one being simpler and analogous. We denote as usual with $\PP$ a generic polyhedron.

\subsection{VEM vertex elements}

We start by recalling briefly the $H^1$-conforming scalar space introduced in \cite{volley}, here generalized to three dimensions. For computing the $L^2-$projection in this case we refer to \cite{projectors}.
Let as usual $k$ be an integer $\ge 1$.
\begin{definition}\label{def:boundspace1}
We define $B_{k}^{\rm vert}(\partial\PP)$ as
the set of functions $v\in C^0(\partial\PP)$ such that
{\color{dblue}
$v_{|e}\in\P_{k}(e)$
on each edge $e\in\partial\PP$, and
on each face $f \in \partial\PP$ it holds $\Delta_2 v_{|f }\in \P_{k-2}(f)$ 
where  $\Delta_2$ is the planar Laplace operator on $f$}.
\end{definition}
\noindent We introduce the family of local vertex spaces $V^{\rm vert}_{3,k}(\PP) \subset H^1(\PP)$ as
\begin{equation}\label{Vhnodal}
V^{\rm vert}_{3,k}(\PP):=\{v|~v_{|\partial\PP}\in B_{k}^{\rm vert}(\partial\PP)
\mbox{ and } \Delta v \in\P_{k-2}(\PP)\} ,
\end{equation}
with the associated set of degrees of freedom:
\begin{eqnarray}
&&{\color{dblue}\bullet} \mbox{ the pointwise value } v(\nu)
\mbox{ for all vertex $\nu$} ,
\label{marzo:vertdof:1}
\\[3pt]
&&{\color{dblue}\bullet} \int_e{v}\,p_{k-2}\,\de
\mbox{ for all edge $e$, for all }
p_{k-2}\in\P_{k-2}(e),
\\[3pt]
&&{\color{dblue}\bullet} \int_{f}{v \, p_{k-2}}\,{\rm d}f
\mbox{ for all face $f$, for all }
p_{k-2} \in\P_{k-2}(f),
\\[3pt]
&&{\color{dblue}\bullet}  \int_{\PP}{v \, p_{k-2}} \dP \quad
\mbox{ for all } p_{k-2} \in\P_{k-2}(\PP) .
\label{marzo:vertdof:4}
\end{eqnarray}
The dimension of the space is thus given by
\begin{equation}\label{added_dim_vert}
\textrm{dim} \big( \Vhv(\PP) \big) = \ell_v + \ell_e \pi_{k-2,1} + \ell_f \pi_{k-2,2} + \pi_{k-2,3} \: .
\end{equation}
As in the above section we can also consider the global spaces. Assuming that we have a polyhedral
domain $\Omega$ and a decomposition $\Th$ in a finite number of polyhedra $\PP$, we can define the global space
{\color{dblue}
\begin{multline}\label{Vhnodalglo}
V^{\rm vert}_{3,k}(\Omega):=\{v\in H^1(\Omega) \mbox{ such that } v_{|\partial\PP}\in B_{k}^{\rm vert}(\partial\PP)\\
\mbox{ and } \Delta v \in\P_{k-2}(\PP) \mbox{ for all elements } \PP\in \Th\} ,
\end{multline}}
with the associated set of degrees of freedom:
\begin{eqnarray}
&&{\color{dblue}\bullet} \mbox{ the pointwise value } v(\nu)
\mbox{ for all vertex $\nu$} ,
\label{marzo:vertdof:1glo}
\\[3pt]
&&{\color{dblue}\bullet} \int_e{v}\,p_{k-2}\,\de
\mbox{ for all edge $e$, for all }
p_{k-2}\in\P_{k-2}(e),
\\[3pt]
&&{\color{dblue}\bullet} \int_{f}{v \, p_{k-2}}{\rm d}f
\mbox{ for all face $f$, for all }
p_{k-2} \in\P_{k-2}(f),
\\[3pt]
&& {\color{dblue}\bullet} \int_{\PP}{v \, p_{k-2}} \dP \quad
\mbox{ for all element } \PP, \mbox{ for all } p_{k-2} \in\P_{k-2}(\PP) .
\label{marzo:vertdof:4glo}
\end{eqnarray}
The dimension of the global space is given by
\begin{multline*}
\!\!\!\!\!\textrm{dim} \big( \Vhv(\Omega) \big) = \{\mbox{number of vertices }\in\Th \} +  \pi_{k-2,1}\times \{\mbox{number of edges }\in\Th\} \\+  \pi_{k-2,2}\times\{ \mbox{number of faces }\in\Th\} + \pi_{k-2,3}\times\{
\mbox{number of elements }\in\Th\} .
\end{multline*}

\subsection{VEM volume elements}

We finally introduce, for all integer $k\ge 0$, the family of volume spaces $\Vhelem{K} (\PP) := \P_{k}(\PP) \subset L^2(\PP)$, with the associated degrees of freedom
$$
\int_{\PP}{v \, p_{k}} \dP \quad \mbox{ for all } p_{k} \in\P_{k}(\PP) .
$$
 This is a actually a space of polynomials (like the ones used, for instance, in Discontinuous Galerkin
 methods), and to deal with it doesn't require any particular care. The corresponding global space will
 be
 \begin{equation}
 \Vhelem{k} (\Omega)=\{v\in L^2(\Omega)\mbox{ such that }v_{|\PP}\in\P_k(\PP)\,\forall\mbox{ element }\PP\in\Th\} .
 \end{equation}

\section{Virtual exact sequences}

We show now that, for the obvious  choices of the polynomial degrees, the set of virtual spaces introduced in this paper constitutes an exact sequence. We start with the (simpler) two-dimensional case.

\begin{theorem}
Let $k \ge 2$, and assume that $\Omega$ is a simply connected polygon, decomposed in a finite
number of polygons $\E$. Then the sequences
\begin{equation}\label{2D:virt:seq1:k}
\R\myarrow{i}
V^{\rm vert}_{2,k}(\Omega)
\myarrow{\bgrad}
V^{\rm edge}_{2,k-1}(\Omega)
\myarrow{\rot}
\P_{k-2}(\Omega)
\myarrow{o}
0
\end{equation}
and
\begin{equation}\label{2D:virt2:seq:k}
\R\myarrow{i}
V^{\rm vert}_{2,k}(\Omega)
\myarrow{\brot}
V^{\rm edge}_{3,k-1}(\Omega)
\myarrow{\div}
\P_{k-2}(\Omega)
\myarrow{o}
0
\end{equation}
are both exact sequences.
\end{theorem}

\begin{proof}
We note first that the two sequences are practically the same, up to a rotation of $\pi/2$. Hence
we will just show the exactness of the sequence \eqref{2D:virt:seq1:k}. Essentially, the only non-trivial
part will be to show that
\begin{itemize}
\item {\bf a.1}
for every $\vv\in V^{\rm edge}_{2,k-1}$ with $\rot\vv=0$ there exists a
$\varphi\in V^{\rm vert}_{2,k}$ such that $\bgrad\varphi=\vv$.
\item {\bf a.2} for every $q\in V^{\rm elem}_{2,k-2}(\Omega)$ there exists a
$\vv\in V^{\rm edge}_{2,k-1}(\Omega)$ such that $\rot\vv=q$.
\end{itemize}

We start with {\bf a.1}. As $\Omega$ is simply connected, we have that  the condition $\rot\vv=0$ implies that  there exist a function $\varphi\in H^1(\Omega)$ such that $\bgrad\varphi=\vv$ { in } $\Omega$. On every edge $e$
of $\Th$ such $\varphi$
will obviously satisfy, as well:
\begin{equation}\label{tancomp}
 \frac{\partial \varphi}{\partial \tt_e}=\vv\cdot\tt_e \in \P_{k-1}(e).
\end{equation}
Then the restriction of $\varphi$ to each $\E\in\Th$ verifies:
\begin{equation}
\varphi_{|e}\in \P_k(e)\;\forall e\in \partial\E;\qquad \Delta\varphi\equiv\div\vv\,\in\P_{k-2}(\E)
\end{equation}
so that clearly $\varphi\in V^{\rm vert}_{2,k}$.

To deal with {\bf a.2}, we first construct a $\bphi$ in $(H^1(\Omega))^2$ such that
$\rot\bphi =q$ and
\begin{equation}
\bphi\cdot\tt=\frac{\int_{\Omega}q\dx}{|\partial\Omega|}\quad\mbox{ on }\partial\Omega,
\end{equation}
where $\tt$ is the unit counterclockwise tangent vector to $\partial\Omega$ and $|\partial\Omega|$
is the length of $\partial\Omega$. Then we consider the element $\vv\in V^{\rm edge}_{2,k-1}(\Omega)$
such that
\begin{equation}
\vv\cdot\tt_e:=\Pi^0_{k-1}(\bphi\cdot\tt_e)\;\forall \mbox{ edge } e \mbox{ in }\Th
\end{equation}
and, within each element $\E$:
\begin{equation}
\rot\vv=\rot\bphi=q, \qquad \div\vv=0.
\end{equation}
Clearly such a $\vv$ solves the problem.
\end{proof}

\begin{remark}\label{closed}

 The construction in the proof of {\bf a.2}  could also be done  if the two-dimensional domain $\Omega$
is a {\it closed surface}, obtained as union of polygons. To fix the ideas, assume that
we deal with the boundary $\partial \PP$ of a polyhedron $\PP$, and that we are given on every face
$f$ of $\PP$ a polynomial $q_f$ of degree $k-2$, in such a way that
\begin{equation}\label{balance}
\sum_{f\in\partial\PP}\int_{f}q_f{\rm d}f=0.
\end{equation}
Then there exists an element $\vv\in B^{\rm edge}_{k-1}(\partial\PP)$ such that on each face $f$ we have
$\rot_2(\vv_{|f})=q_f$. To see that this is true, we define first, for each face $f$, the number
$$\tau_f:=\int_f q_f{\rm d}f.$$
Then we fix, on each edge $e$, an orientation $\tt_e$, we orient each face $f$ with the outward
normal, and we define, {\it for each edge $e$ of $f$}, the counterclockwise tangent unit vector
$\tt^f_c$. Then we consider the {\it combinatorial problem} (defined
 on the topological decomposition $\Th$) of finding for each edge $e$
a real number $\sigma_e$  such that for each face $f$
\begin{equation}
\sum_{e\in\partial f}\sigma_e \tt_e\cdot\tt^f_c=\tau_f.
\end{equation}
This could be solved using the same approach used in the above proof, applied on a flat
polygonal decomposition that is topologically equivalent to the decomposition of
$\partial\PP$ without a face. The last face will fit automatically, due to \eqref{balance}.
Then we take $\vv$ such that on each edge $\vv\cdot\tt\in\P_{k-1}$ with
$\int_e\vv\cdot\tt_e\de=\sigma_e$, and for each face, $\div\vv_f=0$, $\rot\vv_f=q_f $.
\end{remark}

We are now ready to consider the three-dimensional case.
\begin{theorem} Let $k \ge 3$, and assume that $\Omega$ is a simply connected polyhedron, decomposed in a finite number of polyhedra $\PP$. Then the sequence
{\color{dblue}
\begin{equation}\label{3D:virt:seq:k}
\R
\vshortarrow{i}
V^{\rm vert}_{3,k}(\Omega)
\myarrow{\bgrad}
V^{\rm edge}_{3,k-1}(\Omega)
\myarrow{\bcurl}
V^{\rm face}_{3,k-2}(\Omega)
\shortarrow{\div}
\P_{k-3}(\Omega)
\vshortarrow{o}
0
\end{equation}
is exact.}
\end{theorem}
\begin{proof}
It is pretty much obvious, looking at the definitions of the spaces, that
\begin{itemize}
\item a constant function
is in $V^{\rm vert}_{3,k}(\Omega)$ and has zero gradient,
\item the gradient of a function of
$V^{\rm vert}_{3,k}(\Omega)$ is in $V^{\rm edge}_{3,k-1}(\Omega)$ and has zero $\bcurl$,
\item the $\bcurl$  of a vector in $V^{\rm edge}_{3,k-1}(\Omega)$ is in $V^{\rm face}_{3,k-2}(\Omega)$
and has zero divergence,
\item the divergence of a vector of $V^{\rm face}_{3,k-2}(\Omega)$
is in $V^{\rm elem}_{3,k-3}(\Omega)$.
\end{itemize}

Hence, essentially, we have to prove that:
\begin{itemize}
\item {\bf b.1}
for every $\vv\in V^{\rm edge}_{3,k-1}(\Omega)$ with $\bcurl\vv=0$ there exists a
$\varphi\in V^{\rm vert}_{3,k}$ such that $\bgrad\varphi=\vv$.
\item {\bf b.2} for every $\btau\in V^{\rm face}_{3,k-2}(\Omega)$ with $\div\btau=0$ there exists
a $\bphi\in V^{\rm edge}_{3,k-1}(\Omega)$ such that $\bcurl\bphi=\btau$
\item {\bf b.3} for every $q\in V^{\rm elem}_{3,k-3}(\Omega)$ there exists a $\bsigma\in V^{\rm face}_{3,k-2}$ such that $\div\bsigma=q$.
\end{itemize}

The proof of {\bf b.1} is immediate, as in the two-dimensional case $[2.1]$: the function (unique up to a constant) $\vphi$ such that $\bgrad\vphi=\vv$ will verify \eqref{tancomp} on each edge. Moreover, its
restriction $\vphi_f$ to each face $f$ will satisfy $\bgrad_2\vphi=\vv_f$, and so on.

Let us  therefore look at {\bf b.2}. Given $\btau\in V^{\rm face}_{3,k-2}(\Omega)$ with $\div\btau=0$
we first consider (as in Remark \ref{closed}) the element $\bg\in B^{edge}_{k-1}(\partial\Omega)$
such that, on each face $f\in\partial\Omega$
\begin{equation}
\rot_2(\bg_{|f})=\btau\cdot\nn\;(\in \P_{k-2}(f)).
\end{equation}
{\color{dblue} Note that
 \begin{equation}
\sum_{f\in\partial\Omega}\int_f\btau\cdot\nn_{\Omega}^f\df=\int_{\Omega}\div\btau\dO\,=\,0,
\end{equation}
so that the compatibility condition \eqref{balance} is satisfied.}
Then we solve in $\Omega$ the $Div-Curl$ problem
 \begin{equation}\label{divcurl1}
 \div\bpsi=0\quad\mbox{ and }\quad\bcurl\bpsi=\btau\quad \mbox{in $\Omega$},\qquad\mbox{ with }\quad \bpsi_t=\bg
\quad\mbox {on }\partial\Omega .
\end{equation}
The (unique) solution of \eqref{divcurl1} has enough regularity to take the trace of its tangential
component on each edge $e$, and therefore, after deciding an orientation $\tt_e$ for every edge
$e$ in $\Th$, we can take
\begin{equation}
\eta_e:=\Pi^0_{k-1}(\bpsi\cdot\tt_e)\qquad\mbox{ on each edge } e \mbox{ in }\Th.
\end{equation}
At this point, for each element $\PP$ we construct $\bphi\in B^{\rm edge}_{k-1}(\partial\PP)$
by requiring that
{\color{dblue}
\begin{multline}\label{bphiinB}
\bphi\cdot\tt_e=\eta_e\,\mbox{ on each edge, }\ \rot_2 \bphi_f=\btau\cdot\nn_{\PP}^f \mbox{ and }
\div\bphi_f=0\,\mbox{ in each face }f\in\partial\PP.
\end{multline}}
Then we can define $\bphi$ inside each element by choosing, together with \eqref{bphiinB},
\begin{equation}\label{bphiinPP}
 \bcurl\bphi=\btau \mbox{ and } \div\bphi=0\quad\mbox{ in each element }\PP.
\end{equation}
It is easy to see that the boundary conditions given in \eqref{bphiinB} are {\it compatible} with
the requirement $\bcurl\bphi=\btau$, so that the solution of \eqref{bphiinPP} exists. Moreover
it is easy to see that all the necessary orientations fit, in such a way that $\bcurl\bphi$
is globally in $(L^2(\Omega))^3$, so that actually $\bphi\in V^{\rm edge}_{3,k-1}(\Omega)$.

Finally, we have to prove {\bf b.3}. The proof follows very closely the two dimensional case: given
$q\in V^{\rm elem}_{3,k-3}(\Omega)$, we first choose $\bbeta\in (H^1(\Omega))^3$ such that
\begin{equation}
\div \bbeta=q\; \mbox{ in }\Omega\quad \mbox{ and }\bbeta\cdot\nn_{\Omega}=\frac{\int_{\Omega}q\dO}{|\partial\Omega|}
\end{equation}
where, now, $|\partial\Omega|$ is obviously the {\it area} of $\partial\Omega$. Then on each face $f$ of $\Th$
we take
\begin{equation}\label{sigman}
\bsigma\cdot\nn^f=\Pi^0_{k-2}(\bbeta\cdot\nn^f)
\end{equation}
 and inside each element $\PP$ we take $\div\bsigma =q$
and $\bcurl\bsigma =0$. Note again that condition $\div\bsigma =q$ is {\it compatible}
 with the boundary conditions \eqref{sigman} and the orientations will fit in such a way that
 actually $\div\bsigma\in L^2(\Omega)$, so that $\bsigma\in V^{\rm face}_{3,k-2}(\Omega)$.

\end{proof}

\begin{remark} Although here we are not dealing with applications, we point out that, as is
well known (see e.g. \cite{Bossavit:1987}, \cite{Matiussi:1997}, \cite{Hiptmair:2001}, \cite{AFW-Acta}), the exactness of the above sequences are of paramount
importance in proving several properties (as the various forms of {\it inf-sup}, the ellipticity in the kernel,
etc.) that are crucial  in the study of convergence of mixed formulations (see e.g. \cite{Bo-Bre-For}).
\end{remark}

\section{A hint on more general cases}
As already pointed out in the final part of \cite{BFM-mixed} for the particular case of 2D face elements, we observe here that actually in all four cases considered in this paper (face elements and edge elements
in 2D and in 3D), we have at least three parameters to play with in order to create variants of our elements.

For instance, considering the case of 3D face elements, we could choose three different integers
$k_b$, $k_r$ and $k_d$ (all $\ge -1$) and consider, instead of \eqref{def:face:space:3D} the spaces
\begin{equation}\label{def:face:space:3D-v}
\begin{aligned}
V^{\rm face}_{3, {\bf k}}(\PP):=\{& \vv\in H(\div;\PP)\cap H(\bcurl;\PP)  \mbox{ such that: }  \vv\cdot\nn_{\PP}^f\in\P_{k_b}(f)\,\forall \mbox{ face $f$ of } \PP, \\
& \bgrad\div\vv\in\calG_{k_d-1}(\PP),
~\bcurl\vv\in\calR_{k_r}(\PP)\},
\end{aligned}
\end{equation}
where obviously ${\bf k}$ is given by ${\bf k}:=(k_b,k_d,k_r)$. Taking, for a given integer $k$,
the three indices as $k_b=k,\;k_d=k-1,\;k_r =k-1$ we re-obtain  the elements in \eqref{def:face:space:3D}, that
in turn are the natural extension of the BDM $H({\rm div})$-conforming elements. On the other hand,
taking instead $k_b=k,\;k_d=k,\; k_r= k-1$, for $k\ge 0$ we {\color{dblue} would mimic} more the Raviart-Thomas elements.

We also point out that if we know a priori that {\color{dblue}(say, in a mixed formulation)} the {\it vector part} of the solution of our
problem will be a gradient, we could consider the choice $k_b=k,\;k_d=k-1,\; k_r=-1$ obtaining a space that contains all polynomial vectors in $\calG_k$ (that is: vectors that are gradients of some scalar polynomial of degree $\le {\color{dblue}{k+1}}$), a space that is rich enough to provide an optimal approximation of our unknown.

Similarly, for the spaces in \eqref{defVetkP} one can consider the variants

{\color{dblue}
\begin{multline}\label{defVetkP-v}
V^{\rm edge}_{3,{\bf k}}(\PP):=\{\vv|~\vv_{t}\in B_{k_b}^{\rm edge}(\partial\PP),
\div{\vv}\in\P_{k_d}(\PP), \mbox{ and } \bcurl\bcurl\vv\in\calR_{k_r-1}(\PP)\}.
\end{multline}}

On the other hand, for nodal VEMs we can play with two indices, say $k_b$ and $k_{\Delta}$,
to have
\begin{equation}\label{Vhnodal-v}
V^{\rm vert}_{3,{\bf k}}(\PP):=\{v|~v_{|\partial\PP}\in B_{k_b}^{\rm vert}(\partial\PP)
\mbox{ and } \Delta v \in\P_{k_{\Delta}-2}(\PP)\} ,
\end{equation}
and, needless to say, in the definition of $B_{k_b}(\partial\PP)$, the degree of $\Delta_2$ in each face could be different from $k_b$.

 Actually, to be sincere, the amount of possible variants looks overwhelming, and the need of
 numerical experiments (for different applications of practical interest) is enormous.

 \bibliographystyle{spmpsci}      

\end{document}